\documentclass[10pt]{article}
\usepackage{soul}
\usepackage{marginnote}
\usepackage[margin=1in]{geometry} 
\usepackage{graphicx}             
\usepackage{amsmath}              
\usepackage{amsfonts}              
\usepackage{amsthm}     
\usepackage{amssymb}
\usepackage{tikz-cd}          
\usepackage{comment}
\usepackage{color}
\usepackage[colorlinks=false,
					citecolor=green,
					linkbordercolor=red,
					pdfborderstyle=blue, 
					citebordercolor=green, 
					linkcolor=blue]{hyperref}
\usepackage[bottom]{footmisc}

\newcommand{\heis}{\mathbb{H}}
\newcommand{\reals}[1]{\mathbb{R}^{#1}}

\newcommand{\hred}{{\mathbb{H}/\Gamma}}
\newcommand{\sR}{sub-Riemannian }
\newcommand{\xred}{{I^\textnormal{red}}}

\newcommand{\hrede}{{\mathbb{H}/\Gamma^\epsilon}}
\newcommand{\hredee}{{\mathbb{H}/\Gamma^\epsilon_\lambda}}
\newcommand{\hredeee}{{\mathbb{H}/\Gamma^{\epsilon\lambda}}}
\newcommand{\fheis}{{\mathcal{F}_\mathbb{H}}}

\newtheorem*{theorem*}{Theorem}
\newtheorem{theorem}{Theorem}

\newtheorem{lemma}[theorem]{Lemma}
\theoremstyle{definition}\newtheorem{remark}[theorem]{Remark}

\theoremstyle{definition}\newtheorem{definition}[theorem]{Definition}
\newtheorem{corollary}[theorem]{Corollary}
\newtheorem{proposition}[theorem]{Proposition}

\title{Injectivity of the Heisenberg X-ray Transform}
\author{Steven Flynn\footnote{Department of Mathematics, University of California, Santa Cruz, CA 95064; email: spflynn@ucsc.edu}}
\date{October 19th 2020}

\begin{document}
\maketitle
\begin{abstract}
	We initiate the study of X-ray tomography on sub-Riemannian manifolds, for which the Heisenberg group exhibits the simplest nontrivial example. With the language of the group Fourier transform, we prove an operator-valued incarnation of the Fourier Slice Theorem, and apply this new tool to show that a sufficiently regular function on the Heisenberg group is determined by its line integrals over sub-Riemannian geodesics. We also consider the family of taming metrics $g_\epsilon$ approximating the sub-Riemannian metric, and show that the associated X-ray transform is injective for all $\epsilon>0$. This result gives a concrete example of an injective X-ray transform in a geometry with an abundance of conjugate points. 
\end{abstract}

\section{Introduction}
%{\color{red} test  }
Our object of study is the geodesic X-ray transform associated to the sub-Riemannian geometry of the Heisenberg group, which is $\mathbb{H}:=\mathbb{C}\times\mathbb{R}$ with the multiplication law
\begin{align*}
(x+iy, t)(u+iv, s)=\left(x+u+i(y+v), t+s+\tfrac{1}{2}(xv-yu)\right),
\end{align*}
and a metric defined in Section \ref{HeisGeo}. $\heis$ is
the local model for  any 3-dimensional \sR manifold of contact type, in the same sense
that 3-dimensional Euclidean space is the  local model for  any 3-dimensional   Riemannian manifold \cite[Thm. 1]{mitchell1985carnot}. This property positions $\heis$ as the natural homogeneous starting point for studying the integral geometry of contact manifolds, just as Radon first inverted the X-ray transform in $\mathbb{R}^2$.%, whose geodesics project to magnetic and usual geodesics on the base, and these are approximated locally by Heisenberg geodesics [ref]  $\heis$ is also the first step to generalizing Radon's result to the important class of ``non-abelian vector space," called Carnot Groups. 

X-ray transforms, which integrate a function on a manifold over its geodesics, have been extensively studied on Riemannian manifolds and homogeneous spaces \cite{gonzalez2010notes, helgason2010integral}. Helgason showed in \cite{helgason1981x} that the X-ray transform on symmetric spaces of noncompact type is injective. In \cite{klein2009funk} the authors prove injectivity on compact symmetric spaces excluding the n-sphere. Ilmavirta in \cite{ilmavirta2016radon} obtains injectivity on compact Lie groups excluding $S^1$ and $S^3$.  (For a survey of results on Riemanian manifolds with boundary see \cite{ilmavirta20194}.)  To the author's knowledge, X-ray transforms on sub-Riemannian manifolds are virtually unexplored. 

To a function $f\in L^1(\heis)$ we associate the function ${I}f$, its X-ray transform, defined by
\begin{align*}
{I}f(\gamma):=\int f\left(\gamma(s)\right)ds, %\quad (q, p)\in U^*\heis:=H^{-1}\left(\tfrac{1}{2}\right)
\end{align*}
where the geodesics $\gamma$ will be cast as (projections of) integral curves of the Hamiltonian flow on $T^*\heis$ for the degenerate fiber quadratic Hamiltonian later described \eqref{Ham}.
Related integral transforms on $\heis$ have been studied, for example, by Rubin \cite{rubin2012radon}, and Stichartz \cite{strichartz1991lp}, who consider integration over left translates of hyperplanes. We ask whether ${I}f$ determines $f$. 

%The Heisenberg group is not a symmetric space, but it arises as the boundary at infinity of complex hyperbolic space in two dimensions  \cite{capogna2007introduction, thangavelu2012introduction}. 

The sub-Riemannian setting, whose general theory is poorly understood, introduces qualitatively new features to this question. For example, fibers of the unit cotangent bundle $U^*\heis$ (defined in Section  \ref{HeisGeo}) are now cylinders, and there is no unique Levi-Civita connection. Thus $U^*\heis$ has noncompact fibers, 
%unlike three dimensional Riemannian manifolds, but 
and there is no canonical splitting of its tangent space into vertical and horizontal components like there is in the Riemannian case as described in \cite{paternain2012geodesic}. {See} \cite{montgomery2002tour} {for background in sub-Riemannian geometry or} \cite{capogna2007introduction}{ for an extensive introduction to the Heisenberg group.} %One main difference is that we are asking how to reconstruct a function that depends on points of a geodesic, as well as its holonomy. 

%  In the tensor tomography on Riemann surfaces, we may seek to recover a function on the unit tangent bundle from magnetic or usual geodesic curves lifted from the base surface. If they are lifted via horizontal lift, we obtain sub-Riemannian geodesics on the unit sphere bundle. The inhomogenous counterpart of our question is whether one may recover such a function from its integrals over horizontal lifts. 

A standard geometric obstacle to such inverse problems is presented by the presence of conjugate points. In \cite{monard2015geodesic} and \cite{holman2018attenuated} the authors show that %\mn{Superfluous reference removed}
 conjugate points generally inhibit stable inversion of the X-ray transform on Riemannian manifolds, with unconditional loss in two dimensions. Unfortunately, the conjugate points in the Heisenberg group  are ubiquitous; the cut locus to any point passes through that point---a feature generic in sub-Riemannian geometry, where the exponential map is never a local diffeomorphism at the origin \cite[p.222]{strichartz1986sub}. Therefore, standard tools for proving injectivity, such as Pestov energy methods, which typically require a positive-definite second fundamental form \cite{ilmavirta20194} do not apply without a closer look. We prove that, nonetheless, the X-ray transform on the Heisenberg group is injective. 

%Our approach is similar to \cite{agranovsky2019non}, where the authors consider non-geodesic Funk-type transforms on the sphere, and partition the space of curves into sets of concentric circles. We partition geodesics together by their ``charge" as explained in Section \ref{HeisGeo}. Since the charge of Heisenberg geodesics is parameterized by the vertical component of the cylinder fibers in the unit cotangent bundle $U^*\heis$, it is profitable to fix a charge to avoid noncompactness issues.  

A common recipe for inverting such integral transforms is to compute the normal operator $I^*I$, for $I^*$ defined with a suitable measure in the target space, as in \cite{alberti2019unitarization, helgason2010integral, rouviere2001inverting}, and identify the normal operator as a function of distinguished invariant differential operators. 
%$I^*I$ usually has nicer properties; for example, for the Riemannina X-ray transform it is a peudodifferential operator in the absence of conjugate points  \cite{guillemin1979some}.  
On the Heisenberg group, $I^*I$ is not well-behaved (or immediately well-defined on $C_c(\heis)$) due to the singular nature of the sub-Riemannian exponential map, so we focus on studying the transform $I$ directly. We observe a convenient identification of the space of geodesics with a quotient of the Heisenberg group \eqref{geoID}, which allows us to express $I$ as a convolution. We {then} apply the group Fourier transform on the Heisenberg group and its quotient to express $I$ essentially as a multiplication operator (Theorem \ref{theorem2}), from which we deduce that $I$ is injective (Theorem \ref{theorem1}). For background material on the group Fourier transform and the harmonic analysis of the Heisenberg group, see \cite{folland1989harmonic}, \cite{geller1977fourier}, or \cite{thangavelu2012harmonic}.

%Intertwining operator \cite{monard2019functional}

%Microlocal Aspects \cite{guillemin1979some}

%Montgomery, Wong's equation \cite{montgomery1984canonical}

%Calculus of Inv \cite{kakehi1999integral}

%Matrix radon \cite{gonzalez2006invariant}

\section{Main Results}
%We will see that geodesics on the Heisenberg group are exactly horizontal lifts of magnetic geodesics from $\mathbb{R}^2$ to $\heis$ with respect to the contact form $\Theta=dt-\tfrac{1}{2}\left(xdy-ydx\right)$. If $\mathcal{G}$ is the set of unoriented geodesics, then 
The Heisenberg geodesics {exist} for all time and are left-translates of helices and straight lines, as described in Section \ref{HeisGeo}. Let $\mathcal{G}$ be the set of all {maximal} Heisenberg geodesics without orientation, and $\mathcal{G}_{{\lambda}}$ the set of all geodesics having a fixed value ${\lambda}\in\mathbb{R}$, for the ``charge'' ${\lambda}$, which is a constant of motion. 
{We will parameterize $\mathcal{G}_\lambda$ using left-translates of specific model geodesics as in} \cite{ilmavirta2016radon}{, with the caveat that Heisenberg geodesics are not one-parameter subgroups of the Heisenberg group. }

Left translation by any element $(z,t) \in \heis$ is an isometry of $\heis$ and so $\heis$ acts on   $\mathcal{G}$.  This action
does not change the value of ${\lambda}$,
and is a  transitive action  on each leaf $\mathcal{G}_{{\lambda}}$, ${\lambda} \ne 0$.  (It is not transitive on $\mathcal{G}_{0}$ since it does
not change the direction of the line in the plane which the ${\lambda} = 0$ geodesic projects to.) These facts are verified by inspecting the exponential map in \eqref{sRexp}. Thus we can  use $\heis$ to parameterize $\mathcal{G}_{{\lambda}}$, {${\lambda} \ne 0$,}
by fixing a   particular helix $\gamma_{{\lambda}} \in \mathcal{G}_{{\lambda}}$  and left-translating it about. We take this helix to be one  whose projection
is a  circle of radius   $|{R}|= 1/|{\lambda}|$ centered at the origin and parameterized by arclength.  Thus our parameterization of that part
of  $\mathcal{G}$ having ${\lambda} \ne 0$ is 
\begin{equation}
s\mapsto (z, t)\gamma_{{\lambda}}(s), \quad \gamma_{{\lambda}}(s)=\left({R} e^{i(s/{R})}, \tfrac{1}{2}s{R}\right) \in \heis ;\; {R} = 1/\lambda.\label{geod}
\end{equation}

Using this identification, we may parameterize geodesics by $(z, t, \lambda)$ as above, uniquely
%\mn{Indentation added} 
modulo the isotropy group $\Gamma_\lambda:=\{(0, k\pi R^2)\in\heis: k\in\mathbb{Z}\}$ stabilizing $\gamma_\lambda$, and write the X-ray transform concretely as
\begin{align*}
{I}f(z, t, \lambda):={I}_{\lambda}f(z, t):=\int_\mathbb{R}f\left((z, t)\gamma_{{\lambda}}(s)\right)ds, \quad f\in C_c(\heis).
\end{align*}
  We ignore the degenerate case when $\lambda=0$, where the geodesics are straight lines. Furthermore, since $\lambda<0$ corresponds to a $\lambda>0$ geodesic with opposite orientation, we will take $\lambda>0$ {unless otherwise specified}. {Fixing $\lambda>0$, we prove in} Proposition \ref{factorication of I} that $I_\lambda: L^1(\heis)\to L^1(\mathcal{G}_\lambda)$, with a natural measure on the codomain {given in Section} \ref{geodesics}, is well-defined and bounded. {Existing literature (}\cite{ilmavirta2015radon} {and} \cite{ilmavirta2016radon}{for example) profitably considers the X-ray transform as a family of operators indexed by a directional parameter in this way.} 
  
   In \cite[p. 392]{strichartz1991lp}, Strichartz proves indirectly that a function on the Heisenberg group may not in general be recovered from its integrals over $\lambda=0$ geodesics {alone}, but does not consider $\lambda\neq 0$ geodesics. {Indeed, our main result necessarily involves geodesics with nonzero charge $\lambda$:}
% We may now state our main result:
\begin{theorem}
The Heisenberg X-ray transform ${I}:L^1(\heis)\to L^1(\mathcal{G}, d\mathcal{G})$ is injective. In particular, if $f\in L^1(\heis)$, and ${I}_\lambda f=0$ for all ${\lambda}$ in a neighborhood of zero, then $f=0$.\label{theorem1}
\end{theorem}
{The measure $d\mathcal{G}$ on the set of geodesics, $\mathcal{G}$, is defined in Section} \ref{geodesics}.

{Thinking of the charge $\lambda$ as the restricted directional parameter, } {Theorem} \ref{theorem1}{  is an example of limited angle tomography (see} \cite{louis1986incomplete} {and} \cite[Ch. 6]{natterer2001mathematics}).

%\begin{remark}
%	We will actually prove a slightly stronger statement: that $I_{\lambda_k}f=0$ for any sequence of charges $\lambda_1, \lambda_2, ...$ accumulating at zero implies $f=0$. 
%\end{remark}
We prove this result using harmonic analysis adapted to the group structure, modifying familiar results in Euclidean space.  Consider, for example, the Radon and Mean Value Transforms on $\mathbb{R}^2$:
\begin{align}
Rf(s, \theta):=\int_\mathbb{R}f\left(se^{i\theta}+ite^{i\theta}\right)dt, && M^rf(z)=\frac{1}{2\pi}\int_0^{2\pi}f\left(z+re^{i\theta}\right)d\theta \label{radon}
\end{align}
where, say, $f\in C_c(\mathbb{R}^2)$. Taking the Fourier transforms in $s$ and $z$,  respectively, yields
\begin{align*}
\mathcal{F}_{s\mapsto \sigma}{Rf}(\sigma, \theta)=\hat{f}(\sigma e^{i\theta}), && \mathcal{F}_{z\mapsto\zeta}{M^rf}(\zeta)=J_0(r|\zeta|)\hat{f}(\zeta),
\end{align*}
where $J_0$ is the zeroth-order Bessel function \eqref{bessel}. These results are known as Fourier Slice Theorems, or Projection Slice Theorems \cite{natterer2001mathematics}. They reveal that $R$, thought of as a projection onto $\{\theta\}$, becomes a restriction operator onto the ``slice" $\sigma\to\sigma e^{i\theta}$ in the Fourier domain, and that $M^r$ becomes a multiplication operator by $J_0(r|\zeta|)$ when viewed in the Fourier domain. Fourier Slice Theorems exist for more general Radon transforms as well; for example, in \cite{gonzalez2006invariant, kakehi1999integral}.

The Radon and Mean Value Transforms {may be} interpreted as integration over straight lines or magnetic geodesics in Euclidean space. In the case of $\heis$---which is a ``flat" sub-Riemannian geometry---we prove a corresponding Fourier Slice Theorem for Heisenberg geodesics. We use the operator-valued group Fourier transform $\mathcal{F}_\heis$ associated to the Bargmann-Fock representation $\beta_h$ (defined in equation \eqref{BFock}), which has proven a useful tool, for example, by Nachman in \cite{nachman1982wave} to find the fundamental solution for the wave operator for the Heisenberg Laplacian. The theory of $\mathcal{F}_\heis$ is extensively developed in \cite{geller1980fourier, thangavelu2012harmonic}. In particular it has a Plancherel Theorem and Inversion Theorem \cite{folland2016course, geller1977fourier, thangavelu2012harmonic}. 

We identify $\mathcal{G}_\lambda\cong \heis/\Gamma_\lambda$ in Section \ref{geodesics} and so also
%\mn{Added explanation of $\mathcal{F}_{\heis/\Gamma_\lambda}$}
 define in equation \eqref{reducedFourierTransform} the group Fourier transform $\mathcal{F}_{\heis/\Gamma_\lambda}$ on the quotient. We see that in the generalized Fourier domain of $\fheis$ and $\mathcal{F}_{\heis/\Gamma_\lambda}$, the Heisenberg X-ray transform is {essentially} a multiplication operator:

\begin{theorem}[Heisenberg Fourier Slice Theorem]\label{theorem2}
	If $f\in L^1(\heis)$, then
	\begin{align}
	\left(\mathcal{F}_{\heis/\Gamma_\lambda}\left({I}_\lambda f\right)\right)(n)=\left(2\pi/\lambda\right)\mathcal{J}_n\circ \left(\mathcal{F}_\heis f\right)(n\lambda^2), \quad \forall n \in \mathbb{Z}\setminus\{0\}, \forall\lambda>0. \label{hfst}
	\end{align}
\end{theorem} 
%Here $\mathcal{F}_{\heis/\Gamma_\lambda}$ is the ``reduced" Fourier Transform defined in \eqref{reducedFourierTransform}. It is the group Fourier Transform that arises on the reduced Heisenberg group after we make the identification $\mathcal{G}_\lambda\cong \heis/\Gamma_\lambda$ in Section \ref{geodesics}.

{Equation} \eqref{hfst} is an equality of operators acting on Bargmann-Fock space (originally described in \cite{bargmann1961hilbert}),
\begin{align} \mathcal{H}:=\bigg\{F:\mathbb{C}\to\mathbb{C}, \text{ holomorphic}: \frac{1}{\pi}\int_\mathbb{C}|F(\zeta)|^2 e^{-|\zeta|^2}d\zeta<\infty \bigg\}.\label{Fock}
%\mn{Hol$(\mathbb{C})$ changed to ``holomorphic"}
\end{align} 
$\mathcal{J}_n:\mathcal{H}\to\mathcal{H}$ is the operator
\begin{align}
\mathcal{J}_nF(\zeta)=\frac{1}{2\pi i}\left(\frac{1}{en}\right)^{n/2}\oint z^{n-1}e^{-n\zeta/z}F(\zeta+z)dz, \quad n>0 \label{bessel op}%\mn{Corrected an error in the coefficient outside the integral}
\end{align}
where the contour is a circle around the origin oriented counterclockwise (and where $\mathcal{J}_{-n}=\mathcal{J}_n$). Loosely speaking, the Heisenberg X-ray transform $I$ is ``block-diagonalized" in $\lambda$ by the group Fourier transform, and each block is essentially a multiple of $\mathcal{J}_n$. %This Fourier Slice Theorem is interesting because it is saying that on the level of Fourier multipliers, the Heisenberg X-ray transform is essentially integration of a function in a circle through its phase space. 

The classical Fourier Slice Theorem for $R$ in \eqref{radon} states that knowledge of $Rf$ for a fixed $\theta_0$ determines the Fourier transform $\hat{f}(\zeta)$ for all $\zeta\parallel\theta_0$. Similarly, the Heisenberg Fourier Slice Theorem says that knowledge of $I_\lambda f$ for fixed $\lambda$ determines the group Fourier transform $\mathcal{F}_\heis f (h)$, up to multiplication by the operator $\mathcal{J}_n$, for all  $h \in \lambda^2\mathbb{Z}^*$. Therefore, injectivity of $I$ follows once we show that {$\mathcal{J}_n$ is an injective operator at least whenever $n$ is an odd integer} (Proposition \ref{prop1}).
%\begin{proposition}
%	The map $\mathcal{J}_{n}:\mathcal{H}\to\mathcal{H}$ is injective whenever $n$ is an odd integer. %\label{prop1}
%\end{proposition}

Finally, in Section \ref{taming}, we consider the ray transform $I^\epsilon$ (defined in \eqref{epsilonXray}) associated to a special family of left-invariant taming metrics $g_\epsilon$ parameterized by $\epsilon>0$:
\begin{align*}
g_{\epsilon} := dx^2+dy^2+\left(1/\epsilon\right)^{2}\Theta^2; \quad \Theta :=dt-\tfrac{1}{2}(xdy-ydx).
\end{align*}
 {First,} we prove a Heisenberg Fourier Slice Theorem for $g_\epsilon$ geodesics:
\begin{theorem}[$g_\epsilon$ Heisenberg Fourier Slice Theorem]\label{slice} If $f\in L^1(\heis)$, and $\epsilon>0$ then%\mn{Moved here from Section 5}
	\begin{align*}
	\mathcal{F}_\hredee\left(I_\lambda^\epsilon f\right)(n)=(2\pi/\lambda)\mathcal{J}_n\left(\frac{1}{\sqrt{1+2\epsilon^2\lambda^2}}\right)\circ\left(\fheis f\right)\left(\frac{n\lambda^2}{1+2\epsilon^2\lambda^2}\right), \quad \forall n\in\mathbb{Z}^*, \;\forall \lambda>0.
	\end{align*}
\end{theorem}
{Here $\mathcal{J}_n(r)$, $r>0$, is defined in} \eqref{besselop3}.

We then use Theorem \ref{slice} in the same way with Proposition \ref{propInjective} to show that $I^\epsilon$ is injective:
\begin{theorem}\label{theorem1epsilon}%\mn{Moved here from Section 5}
	For all $\epsilon>0$, the Heisenberg taming X-ray transform $I^\epsilon: L^1(\heis)\to L^1(\mathcal{G}^\epsilon, d\mathcal{G}^\epsilon)$ is injective. In particular, if $f\in L^1(\heis)$ and $I_\lambda^\epsilon f=0$ for all $\lambda$ in a neighborhood of zero, then $f=0$. 
\end{theorem}
{The measure $d\mathcal{G}^\epsilon$ on the set of $g_\epsilon$-geodesics, $\mathcal{G}^\epsilon$, is defined in} \eqref{measureEpsilon2}.

The first part of Theorem \ref{theorem1epsilon} is not new. In \cite{peyerimhoff2018support} the authors prove a %\mn{Replaced the last paragraph by this one describing the relationship between Theorem 4 and a reference I recently found}
 support theorem for geodesics of left-invariant metrics on the Heisenberg group, which implies injectivity of the associated X-ray transform. However, to the author's knowledge, the second part of Theorem \ref{theorem1epsilon} is new. 

%n \cite{peyerimhoff2016x}, the authors prove that the X-ray transform is injective on step-2 Nilpotent Lie groups of rank greater than 3 with a left-invariant metric. However, this result assumes the existence of 2-dimensional totally geodesic surfaces, which excludes the Heisenberg group.
\section{Preliminaries}
\subsection{Heisenberg Geometry}\label{HeisGeo}
We define the sub-Riemannian metric on $\heis$ by declaring the left-invariant vector fields 
\begin{align}
X=\partial_x-\frac{1}{2}y\partial_t, && Y=\partial_y+\frac{1}{2}x\partial_t, \label{XY}
\end{align}
to be orthonormal, and  the length of  $T=\partial_t$ to be  infinite. Then any finite length smooth path in $\heis$ must be 
tangent to the  nonintegrable distribution $\mathcal{D}_q:=\text{Span}\{X_q, Y_q\}$, $q\in \heis$.    We call such a path {\it horizontal}.   
The length of a horizontal path equals the length of its projection to the plane by the map
$$\pi(x,y,t) = (x,y).$$
A  minimizing Heisenberg geodesic is a shortest
horizontal path joining two points of $\heis$. That any two points in $\heis$ are connected by a horizontal path is guaranteed by Chow's Theorem and the fact that $\mathcal{D}$ satisfies the H\"{o}rmander condition (i.e. $\mathcal{D}$ is bracket-generating). 

The  fiber quadratic Hamiltonian $H: T^*\heis \to \mathbb{R}$ given in canonical coordinates by 
\begin{align}
H(x,y,t,p_x, p_y, p_t) =\tfrac{1}{2}\left((p_x-\tfrac{1}{2}y{p_t})^2+(p_y+\tfrac{1}{2}x{p_t})^2\right)
\label{Ham}%\mn{;, removed}
\end{align}
%\todo{;, removed}\\
%($H$ is  the principal symbol of the subLaplacian $\Delta = X^2 + Y^2$. ) 
generates  the Heisenberg   geodesics.  By `generate' we mean that  any   solution to   Hamilton's equations for $H$ 
projects, via the canonical projection  $T^* \heis \to \heis$, to a   \sR geodesic,
and conversely, all Heisenberg  geodesics arise   this way \cite[Sec 1.5]{montgomery2002tour}.   If we want   geodesics   parameterized by arclength we only
take solutions for which $H=1/2$. (Thus, we define the unit cotangent bundle $U^*\heis$ as the set of all $(q, p)\in T^*\heis$ for which $H(q, p)=1/2$.)   
These  geodesics can be best  understood by their projection under $\pi$ to the plane:   they are circles or lines.
Indeed 
\begin{align*}
\dot{p}_t=-\frac{\partial H}{\partial t}=0,
\end{align*}
so that $\lambda:=p_t$ is a constant of motion. 
If we interpret ${\lambda}$  as the charge of a particle, then $H$, viewed as%\mn{Removed reference to the elementery charge ``e"}
a Hamiltonian on $T^*\reals{2}$, is the Hamiltonian for a particle of charge ${\lambda} $ travelling in the plane
under the influence of a  constant unit strength magnetic field.  These solutions are well-known and easy to derive \cite[p. 12]{montgomery2002tour} .
When $H =1/2$ they are circles of radius ${R}=1/|{\lambda}|$ for ${\lambda} \ne 0$, and lines when ${\lambda} =0$.   See eq \eqref{geod} 
for a concrete representation of all geodesics with ${\lambda}  \ne 0$. 

\subsection{The group Fourier transforms}
We start by giving a brief description of the representation theory of the Heisenberg group. A more detailed discussion can be found in \cite{folland1989harmonic}. Denote by $\mathcal{U}\left(\mathcal{H}\right)$ the set of unitary operators on Bargmann-Fock space, defined in \eqref{Fock}. For each ${h} \in \mathbb{R}^*=\mathbb{R}\setminus \{0\}$, the map (motivated in Section \ref{infinitesimal})
\begin{align*}
\beta_{h} : \heis \to \mathcal{U}\left(\mathcal{H}\right)
\end{align*}
given by
\begin{align}
\beta_{h}(z, t)F(\zeta):=
e^{2i{h} t - \sqrt{h}\zeta\overline{z}-\frac{{h}}{2}|z|^2}F(\zeta+\sqrt{h}z), \quad F\in \mathcal{H}, \;{h}>0, \label{BFock}%\mn{$F\in \mathcal{H}$ added.}
%\beta_{h}(z, t)u(\zeta)&:=e^{i{h} t - \sqrt{\frac{-{h}}{2}}\zeta\overline{z}+\frac{{h}}{4}|z|^2}u(\zeta+\sqrt{\tfrac{-{h}}{2}}\overline{z}), \quad {h}<0 
\end{align}
and $\beta_{h}(z, t)=\beta_{|{h}|}(\overline{z}, -t)$ for ${h} <0$,
is a strongly continuous unitary representation of the Heisenberg group on $\mathcal{H}$. Moreover, it is known that these representations are irreducible, and by the Stone-von Neumann Theorem, up to unitary equivalence, these are all of the irreducible unitary representations on $\heis$ that are nontrivial on the center of $\heis$ \cite{folland1989harmonic}. 

We define the group Fourier transform of an integrable function on $\heis$. Denote by $\mathcal{B}(\mathcal{H})$ the space of bounded operators on $\mathcal{H}$. The Heisenberg Fourier transform of $f\in L^1(\heis)$ is the operator-valued function
\begin{align*}
&\mathcal{F}_\heis f:\mathbb{R}^*\to\mathcal{B}(\mathcal{H}) \\
&\mathcal{F}_\heis f({h}):=\int_\heis f(q)\beta_{h}(q)^*dq
\end{align*}
where the integral is taken in the Bochner sense \cite[p. 11]{thangavelu2012harmonic}. Think of $h$ as a semi-classical parameter.%; indeed, this converges to the classical Fourier transfomr on $\mathbb{R}^2$ as $h\to 0$\textcolor{cyan}{REFRERENCE}.
\begin{remark}
	Many authors define $\mathcal{F}_\heis$ alternatively with the Schr\"{o}dinger representations. Our definition seems more natural for studying the X-ray transform due to the simplicity of \eqref{bessel op}, and is equivalent by conjugation with a unitary intertwining map; the choice is largely a personal preference. We also normalize the representations $\beta_h$ in such a way that they all act on the same space $\mathcal{H}$, rather than a family of spaces parameterized by $h\in\mathbb{R}^*$, as in  \cite{folland1989harmonic}. 
\end{remark}
If $f\in L^1(\heis)\cap L^2(\heis)$, then $\mathcal{F}_\heis(f)({h})$ is a Hilbert-Schmidt operator on $\mathcal{H}$ \cite{geller1977fourier, geller1980fourier}. Let $S_2$ denote the space of Hilbert-Schmidt operators on $\mathcal{H}$, and define the Hilbert Space $L^2(\mathbb{R}^*, S_2; d\mu)=L^2(S_2)$ via the inner product
\begin{align*}
\langle A, B\rangle_{L^2(S_2)}:=\int_{\mathbb{R}^*}\text{tr}\left(A({h})B({h})^*\right)d\mu({h}), \quad d\mu=\pi^{-2}|{h}|d{h}.
\end{align*}

{We will need the following theorems from Geller, normalized to account for the slightly different group law for $\heis$ used here and in} \cite{geller1977fourier}. 
\begin{theorem}[{\cite[Plancherel Theorem]{geller1977fourier}}]%\mn{Numbering and references added to the following two theorems} 
	If $f\in L^1(\heis)\cap L^2(\heis)$, then $||f||_{L^2(\heis)}=||\mathcal{F}_\heis f||_{L^2(S_2)}$.
\end{theorem}
\begin{theorem}[{\cite[Fourier Inversion Theorem]{geller1977fourier}}]
	If $f\in S(\heis)$, Schwartz space on $\mathbb{R}^3$, then 
	\begin{align}
	f(q)=\int_{\mathbb{R}^*}\textnormal{tr}\left(\beta_{h}(q)\mathcal{F}_\heis f({h})\right)d\mu({h}), \quad q\in \heis.\label{inversionFormula}
	\end{align}
\end{theorem}
\noindent 
Thus $\mathcal{F}_\heis$ extends to an isometry from $L^2(\heis)$ into $L^2(S_2)$. In fact, it is onto as well. Furthermore if $f\in L^1(\heis)$, then by convolving $f$ with an approximation of identity, we may use $\eqref{inversionFormula}$ to prove $\mathcal{F}_\heis$ is injective on $L^1(\heis)$.

While the definition above is sufficient for our purposes, we remark that $\mathcal{F}_\heis$ has been extended to much more general classes of function such as tempered distributions \cite{bahouri2018tempered}. In \cite{bahouri2012phase, taylor1984noncommutative}, and much more generally in \cite{kammerer2019semi} the authors use the group Fourier transform to develop theory of pseudo-differential operators.

Finally, $\Gamma_\lambda:=\{(0, k \pi R^2 )\in\heis : k\in\mathbb{Z}\}$, where $R=1/\lambda$, is a discrete subgroup of the center of $\heis$. Since $\beta_h(z, t)=e^{2iht}\beta_h(z, 0)$, the representation $\beta_{h}$ descends to the so-called reduced Heisenberg group $\heis/\Gamma_\lambda$ if and only if ${h}\in \lambda^2 \mathbb{Z}^*$.  To a function $g\in L^1(\heis/\Gamma_\lambda)$, we associate the \textit{reduced Fourier transform}, defined as
\begin{align}
&\mathcal{F}_{\heis/\Gamma_\lambda}(g): \mathbb{Z}^*\to \mathcal{B}\left(\mathcal{H}\right)\nonumber \\
&\mathcal{F}_{\heis/\Gamma_\lambda}(g)(n):=\int_{\heis/\Gamma_\lambda}g(q)\beta_{n\lambda^2}(q)^*dq, \label{reducedFourierTransform}
\end{align}
where $\mathbb{Z}^*:=\mathbb{Z}\setminus\{0\}$. 
\begin{remark}
	The reduced Fourier transform defined above is not invertible unless we  also consider the 
	%\mn{$f$ changed to $g$ to remain consistent with the previous definition.}
	 representations $(z, t)\mapsto e^{iz\cdot\xi}; \, \xi\in\mathbb{C}$, which are trivial on the center, in the definition. (Indeed, if $\partial_t g(z, t)=0$, then $\mathcal{F}_{\heis/\Gamma_\lambda}g=0$.) This extension is not necessary for our purposes. 
\end{remark}
\section{Proof of Theorems \ref{theorem1} and  \ref{theorem2}}
\subsection{The space of geodesics}\label{geodesics}
Recall that $\heis$ acts transitively on $\mathcal{G}_{{\lambda}}$ on the left. Since
\begin{align}
(0, \pi {{R}^2})\gamma_{{\lambda}}(s)=\left({R} e^{is/{R}}, \tfrac{{R}}{2}(s+2\pi{R})\right)=\gamma_{{\lambda}}\left(s+2\pi {{R}}\right); \;{R}=1/\lambda,\label{stab}
\end{align}
the subgroup $\Gamma_\lambda:=\{(0, k\pi{R}^2)\in\heis: k\in\mathbb{Z}\}$ stabilizes $\mathcal{G}_{{\lambda}}$. Upon fixing $\gamma_\lambda$, we have the identification
\begin{align}
\mathcal{G}_{{\lambda}}&\cong\heis/\Gamma_\lambda\nonumber\\
(z, t)\gamma_{{\lambda}}&\mapsto (z, t)\Gamma_\lambda. \label{geoID}
\end{align}
When $\lambda=1$, we omit subscripts and write $\Gamma=\Gamma_1$.

Let ${d\mu_\lambda(z, t)}\cong dx\wedge dy\wedge dt$ be the Haar measure on $\heis/\Gamma_\lambda$, and let  $\mathcal{G}_{{\lambda}}$ inherit a multiple of the Haar measure, $d\mathcal{G}_{{\lambda}}:={\lambda} dx\wedge dy\wedge dt$, normalized to satisfy \eqref{subSantalo}.  Furthermore, let $d\mathcal{G}:=\lambda e^{-\lambda} dx\wedge dy\wedge dt\wedge d\lambda$, {with a weight} chosen to ensure boundedness in Proposition \ref{factorication of I}.
%which agrees with the Liouville measure inherited by $\mathcal{G}$ when viewed as quotient of cotangent bundle $T^*\heis$ by the geodesic flow. 
\subsection{Simplification to the reduced X-ray transform}
The dilation map, $\delta_\lambda(z, t):=(\lambda z, \lambda^2t)$, is an automorphism of the Heisenberg group for $\lambda\neq 0$. Furthermore,
\begin{align*}
\delta_\lambda: \Gamma_\lambda \ni (0, k\pi\lambda^{-2})\mapsto (0, k \pi)\in \Gamma,
\end{align*} 
so $\delta_\lambda: \heis/\Gamma_\lambda\to \heis/\Gamma$ is well-defined. Denote by $\delta_\lambda^*$ the pullback {operator (sometimes called the pullback relation)} defined on functions: 
\begin{align*}
\delta_\lambda^*:&L^1(\heis)\to L^1(\heis) &\delta_\lambda^*f(z, t)&=f(\lambda z, \lambda^2 t) \\
\delta_\lambda^*:&L^1(\heis/\Gamma)\to L^1(\heis/\Gamma_\lambda) &\delta_\lambda^*g\left((z, t)\Gamma_\lambda\right)&=g\left((\lambda z, \lambda^2 t)\Gamma\right).
\end{align*}
\begin{remark}
	In the sequel, we write any function $g: \heis/\Gamma_\lambda\to \mathbb{C}$ as $g(z, t)$, in place of $g\left((z, t)\Gamma_\lambda\right)$, understanding that the $t$ variable is taken mod $\pi\lambda^{-2}$.
\end{remark}
The dilation map $\delta_\lambda$ is relevant because it is a conformal map for the sub-Riemannian metric (with constant conformal factor $\lambda$). Consequently, we have the following homogeneity of the ray transform:
\begin{proposition}[Homogeneity of $I$] \label{propHomo}For $f\in C_c(\heis)$, 
	\begin{align}
	{I}_\lambda f(z, t)=(1/\lambda) \delta^*_{\lambda}\left({I}_1(\delta^*_{1/{\lambda}} f)\right)(z, t). \label{homo}
	\end{align}
\end{proposition}
\begin{proof}
	Note that dilation preserves geodesics but rescales their speed: 
	\begin{align}
	\delta_{1/\lambda}\gamma_1(s)=\gamma_{\lambda}(s/\lambda). \label{rescale1}
	\end{align} 
	Then
	\begin{align*}
	 \delta^*_{\lambda}\left({I}_1(\delta^*_{1/{\lambda}} f)\right)(z, t)
	&=I_1\left(\delta^*_{1/\lambda}f\right)(\lambda z, \lambda^2 t) \\
	&=\int_\mathbb{R}\delta^*_{1/\lambda}f\left((\lambda z, \lambda^2 t)\gamma_1(s)\right)ds \\
	&=\int_\mathbb{R}f\left(\delta_{1/\lambda}(\lambda z, \lambda^2 t)\delta_{1/\lambda}\left(\gamma_1(s)\right)\right)ds, &\text{because }\delta_\lambda\in\text{Aut}(\heis),\\
	%&=\int_\mathbb{R}f\left((z, t)\delta_{1/\lambda}\left(\gamma_1(s)\right)\right)ds\\
	&=\int_\mathbb{R}f\left((z, t)\gamma_\lambda(s/\lambda)\right)ds, &\text{by } \eqref{rescale1},\\
	&=\lambda \int_\mathbb{R}f\left((z, t)\gamma_\lambda(s)\right)ds=\lambda I_\lambda f(z, t).
	\end{align*}
\end{proof}
Next, we exploit the periodic symmetry of Heisenberg geodesics to reduce the X-ray transform to one period.
\begin{proposition} For any $\lambda>0$, ${I}_\lambda:L^1(\heis)\to L^1(\mathcal{G}_{{\lambda}})$ is well-defined, bounded, and factors in the following way:%\mn{Reminder of $\mathcal{G}_{{\lambda}}\cong \heis/\Gamma_\lambda$ added.}
	\[ 
	\begin{tikzcd}
	L^1(\heis) \arrow[d, "P_{\lambda}" left=1]  \arrow[r, "{I}_\lambda"]  &L^1(\mathcal{G}_{{\lambda}}\cong \heis/\Gamma_\lambda)\\
	L^1(\heis/\Gamma_\lambda) \arrow[ur, "{I}^\textnormal{red}_{\lambda}" below=10, pos=0.75]
	\end{tikzcd}
	\]
	where the maps which we call \textit{Central Periodization} and the \textit{reduced X-ray transform} are given by
	\begin{align*}
	P_{\lambda}f\left(z, t\right)=\sum_{k\in\mathbb{Z}}f\left(z, t+k\pi{R}^2\right), && {I}^\textnormal{red}_{\lambda}g(z, t)=\int_0^{2\pi {R}}g\left((z, t)\gamma_{{\lambda}}(s)\right)ds; \, {R}=1/\lambda.
	\end{align*}
	Furthermore, ${I}:L^1(\heis)\to L^1(\mathcal{G},d\mathcal{G})$ is well-defined and bounded.\label{factorication of I}
\end{proposition}
\begin{proof}
	By homogeneity \eqref{homo}, and since pullback by $\delta_\lambda$ is bounded in the above $L^1$ spaces for $\lambda\neq 0$, it suffices to prove the proposition for $\lambda=1$. For this case, we omit subscripts and write $P$ and ${I}^\text{red}$.
	The map
	\begin{align*}
	C_c(\heis)\ni f  \mapsto \int_{\heis/\Gamma}{P}f\left(z, t\right){d\mu_1(z, t)}
	\end{align*}
	is a left-invariant positive linear functional on $ C_c(\heis)$. By uniqueness of the Haar measure on $\heis$ (which is just the Lebesgue measure), and the  Riesz-Representation theorem, $\exists c>0$ such that
	\begin{align}
	\int_{\heis/\Gamma}{P}f\left(z, t\right){d\mu_1(z, t)}=c\int_\heis f(z, t)d(z, t), \label{poisson}
	\end{align}
	and one may check that $c=1$ (see \cite[Thm. 2.49]{folland2016course} for the general statement). 
%	\begin{align*}
%	\int_{\heis/\Gamma_R}P_{\Gamma_R}f\left((z, t)\Gamma_R\right){d\mu_1(z, t)}_R 
%	&=\int_\mathbb{C}\int_0^{\pi {{R}^2}}\sum_{k\in\mathbb{Z}}f(z+k\pi {{R}^2})dtdxdy \\
%	&=\int_\mathbb{C}\sum_{k\in\mathbb{Z}}\int_0^{\pi {{R}^2}}f(z+k\pi {{R}^2})dtdxdy \\
%	&=\int_\mathbb{C}\sum_{k\in\mathbb{Z}}\int_{k\pi {{R}^2}}^{(k+1)\pi {{R}^2}}f(z, t)dtdxdy \\
%	\end{align*}
	So in particular, $||{P}f||_{L^1(\heis/\Gamma)}\leq ||f||_{L^1(\heis)}.$
	
	For $g\in C_c(\heis/\Gamma)$, 
	\begin{align*}
	||{I}^\textnormal{red}g||_{L^1(\mathcal{G}_1)}
	&=\int_{\mathcal{G}_1}|{I}^\textnormal{red}g\left(z, t\right)|d\mathcal{G}_1\\
	&=\int_{\mathbb{H}/\Gamma}\bigg|\int_0^{2\pi}g\left((z, t)\gamma_1(s)\right)ds\bigg|{d\mu_1(z, t)}\\
	&\leq \int_0^{2\pi}\int_{\heis/\Gamma}|g\left((z, t)\gamma_{1}(s)\right)|{d\mu_1(z, t)} ds\\
	&=\int_0^{2\pi}\int_{\heis/\Gamma}|g\left((z, t)\right)|{d\mu_1\left((z, t)\gamma_{1}(s)^{-1}\right)} ds\\
	&=\int_0^{2\pi}\int_{\heis/\Gamma}|g\left((z, t)\right)|{d\mu_1(z, t)} ds, && \text{since }\heis/\Gamma\text{ is unimodular, (i.e. } \mu_1 \text{ is bi-invariant)}\\
	&=2\pi||g||_{L^1(\heis/\Gamma)}.
	\end{align*}
	Thus ${P}$ and ${I}^\textnormal{red}$ extend to $L^1$ bounded maps. Given $f\in C_c(\heis)$, since $Pf\in C_c(\hred)$ and
	\begin{align*}
	{I}^\textnormal{red}{P}f(z, t)
	&=\int_0^{2\pi}\sum_{k\in\mathbb{Z}}f\left((z, t+k\pi)\gamma_1(s)\right) ds 
	=\int_0^{2\pi}\sum_{k\in\mathbb{Z}}f\left((z, t)\gamma_{1}(s+2\pi  k)\right)ds, \quad \text{by }\eqref{stab}, \\
	&=\sum_{k\in\mathbb{Z}}\int_0^{2\pi}f\left((z, t)\gamma_{1}(s+2\pi  k)\right)ds
	=\sum_{k\in\mathbb{Z}}\int_{2\pi k}^{2\pi (k+1)}f\left((z, t)\gamma_{1}(s)\right)ds={I}_1f(z, t),
	\end{align*}
	we have $||{I}_1f||_{L^1(\mathcal{G}_1)}\leq 2\pi||f||_{L^1(\heis)}$. The third equality follows from uniform convergence of the integrand on the interval $[0, 2\pi]\ni s$. Therefore $I_1$ extends to a bounded map from $L^1(\heis)$ to $L^1(\mathcal{G}_1)$. In particular one may check, using \eqref{homo}, that $||I_\lambda f||_{L^1(\mathcal{G}_\lambda)}=||I_1f||_{L^1(\mathcal{G}_1)}\leq 2\pi||f||_{L^1(\heis)}$.
	
	Finally, for $f\in L^1(\heis)$, we have
	\begin{align*}
	||If||_{L^1(\mathcal{G})}:=&\int_\mathcal{G}|If(z, t, \lambda)|d\mathcal{G}\\
	=&\int_0^\infty\int_\mathcal{G_\lambda}|I_\lambda f(z, t)|d\mathcal{G}_\lambda e^{-\lambda}d\lambda\\
	\leq& 2\pi||f||_{L^1(\heis)} \int_0^\infty e^{-\lambda}d\lambda=2\pi||f||_{L^1(\heis)}%\mn{Period removed}
	\end{align*}
	as desired.
\end{proof}
\begin{remark}
	The reduced X-ray transform $I^\text{red}: L^1(\heis/\Gamma)\to L^1(\mathcal{G}_1)$ is not injective. In fact, if
	\begin{align*}
	g\left(z, t\right)=z^2e^{-|z|^2}e^{4it},
	\end{align*}
	then $I^\textnormal{red}g=0$. In Appendix \ref{redSVD} we give essentially a Singular Value Decomposition of $I^\textnormal{red}$ {and characterize its kernel on $L^2(\heis/\Gamma)$.}
\end{remark}
\begin{remark}
	From these computations, we may also deduce a sub-Riemannian Santal\'{o} formula:
	\begin{align}
	\int_{\mathcal{G}_\lambda}I_\lambda f(z, t)d\mathcal{G}_\lambda=2\pi\int_\heis f(z, t)d(z, t), \quad f\in L^1(\heis).\label{subSantalo}
	\end{align}
	This is an example of a Santal\'{o} formula like those proven in \cite{prandi2015sub}, but without the latter's restriction to the ``reduced unit cotangent bundle."
\end{remark}
\subsection{Lemmas on the group Fourier transform}
We now prove a few general properties of the group Fourier transform. The first is a Poisson Summation Formula for $\heis\to\heis/\Gamma$ - a quick consequence of the classical version. The author has not found a reference for this version, but does not believe it is new. 
\begin{lemma}[Poisson Summation Formula]\label{summationformula}%\mn{Space removed}
	If $f\in L^1(\heis)$, then
	\begin{align*}
	\mathcal{F}_{\heis/\Gamma}\left({P}f\right)(n)=\mathcal{F}_\heis f\left(n\right), \quad \forall n\in\mathbb{Z}^*.
	\end{align*}
\end{lemma}
\begin{proof}
	For $F, G \in \mathcal{H}$,
	\begin{align*}
	\mathcal\langle\mathcal{F}_{\heis/\Gamma}\left({P}f\right)(n)F, G\rangle_\mathcal{H}:=&\int_{\heis/\Gamma}\sum_{k\in\mathbb{Z}}f(z, t+k\pi)\langle\beta_{n}(z, t)^*F, G\rangle_\mathcal{H}{d\mu_1(z, t)} \\
	=&\int_{\heis/\Gamma}\sum_{k\in\mathbb{Z}}f(z, t+k\pi )\langle\beta_{n}(z, t+k\pi )^*F, G\rangle_\mathcal{H}{d\mu_1(z, t)}, \quad\text{since }\beta_{n}(z, t)=e^{2int}\beta_{n}(z, 0), \\
	=&\int_\heis f(z, t)\langle\beta_{n}(z, t)^*F, G\rangle_\mathcal{H}d(z, t),
	\end{align*}
	where the third equality follows from \eqref{poisson}, and the fact that $f(z, t)\langle\beta_{n}(z, t)^*F, G\rangle_\mathcal{H}\in L^1(\heis)$ by the Cauchy-Schwartz inequality. Since $F$ and $G$ were arbitrary, the identity follows from the definition of the Bochner integral.
\end{proof}
Next, we observe how the Fourier transforms behave with respect to dilations.
\begin{lemma}[Dilation Property]
	Fix $\lambda>0$.\\ If $f\in L^1(\heis)$, then
	\begin{align*}
	\mathcal{F}_\heis\left(\delta_\lambda^*f\right)({h})=\lambda^{-4}\mathcal{F}_\heis f({h}/\lambda^2), \quad \forall {h}\in\mathbb{R}^*.
\end{align*}
 And if $g\in L^1(\heis/\Gamma)$, then
\begin{align*}
	\mathcal{F}_{\heis/\Gamma_\lambda}\left(\delta_\lambda^*g\right)(n)&=\lambda^{-4}\mathcal{F}_{\heis/\Gamma}(g)(n), \quad \forall n\in\mathbb{Z}^*.
	\end{align*}\label{dilation lemma}
\end{lemma}
\noindent We expect the above exponent of $\lambda$ because the homogeneous dimension of the Heisenberg group is $4$. 
\begin{proof}
	\begin{align*}
	\mathcal{F}_\heis\left(\delta_\lambda^*f\right)({h})
	&=\int_\heis f(\lambda z, \lambda^2 t)\beta_h(z, t)^*d(z, t)=\lambda^{-4}\int_\heis f(z, t)\beta_h\left(\lambda^{-1}z, \lambda^{-2}t\right)^*d(z, t)\\
	&=\lambda^{-4}\int_\heis f(z, t)\beta_{h/\lambda^2}(z, t)^*d(z, t)=\lambda^{-4}\mathcal{F} _\heis f(h/\lambda^2),
	\end{align*}
	and the proof for $\mathcal{F}_{\heis/\Gamma}$ is nearly identical. %\mn{``of" removed}
\end{proof}
\subsection{Proof of Theorem \ref{theorem2}}
The reduced X-ray transform $I^\text{red}$ is equivariant with respect to left translation by $\heis$ in the sense that 
\begin{align*}
\xred \left(L^*_{(w, s)}g\right)(z, t)
&=\int_0^{2\pi}L^*_{(w, s)}g\left((z, t)\gamma_1(\theta)\right)d\theta
=\int_0^{2\pi}g\left((w, s)(z, t)\gamma_1(\theta)\right)d\theta\\
&=\xred g\left((w, s)(z, t)\right)=\left(L^*_{(w, s)}\xred g\right)(z, t).
\end{align*}Thus, $\xred$ is a convolution operator. In fact, if we define the compactly supported distribution $\kappa\in\mathcal{E}'(\hred)$ by $\kappa(g):= \int_0^{2\pi}g\left(\gamma_1(\theta)^{-1}\right)d\theta$ then $\xred g=\kappa*g$, where $f*g\left(z, t\right):=\int_\hred f\left((z, t)(w, s)^{-1}\right)g(w, s)d(w, s)\Gamma$.  Therefore, by an analogous Paley-Wiener theory \cite[Ch.1]{thangavelu2012harmonic}, we expect $\mathcal{F}_\hred(\kappa)(n)\in \mathcal{B}\left(\mathcal{H}\right)$, and $\mathcal{F}_\hred\left({\xred g}\right)(n)=\mathcal{F}_\hred(\kappa)(n)\circ\mathcal{F}_\hred(g)(n)$. The next proposition makes this heuristic explicit.
\begin{proposition}
	If $g\in L^1(\heis/\Gamma)$, then for all $n\in\mathbb{Z}^*$, 
	\begin{align*}
	\mathcal{F}_{\heis/\Gamma}\left({I}^\textnormal{red}g\right)(n)={(2\pi)}\mathcal{J}_n\circ\mathcal{F}_{\heis/\Gamma}(g)(n)%\mn{Period removed}
	\end{align*}
	with $\mathcal{J}_n$ defined in \eqref{bessel op}.
	\begin{proof}
		\begin{align*}
		\mathcal{F}_{\heis/\Gamma}\left({I}^\textnormal{red}g\right)(n)
		:=&\int_{\heis/\Gamma}\int_0^{2\pi }g\left((z, t)\gamma_1(s)\right)\beta_{n}(z, t)^*ds{d\mu_1(z, t)}\\
		=&\int_0^{2\pi}\int_{\heis/\Gamma}g\left(z, t\right)\beta_{n}\left((z, t)\gamma_1(s)^{-1}\right)^*{d\mu_1(z, t)} ds, & \text{since }\heis/\Gamma\text{ is unimodular,}\\
		=&\int_0^{2\pi }\int_{\heis/\Gamma}g\left(z, t\right)\beta_{n}\left(\gamma_1(s)\right)\circ\beta_{n}(z, t)^*{d\mu_1(z, t)} ds, &\text{since }\beta_{n}(z, t)\text{ is a unitary rep},\\
		=&\int_0^{2\pi}\beta_{n}\left(\gamma_1(s)\right)ds\circ\int_{\heis/\Gamma}g\left(z, t\right)\beta_{n}(z, t)^*{d\mu_1(z, t)}\\		
		=&(2\pi)\mathcal{J}_n\circ \mathcal{F}_{\heis/\Gamma}\left(g\right)(n)%\mn{Peroid removed}
		\end{align*}
		%and
		%\begin{align*}
		%\int_0^{2\pi}\beta_{2n}\left(\gamma_1(s)\right)ds=2\pi \mathcal{J}_n.
		%\end{align*}
		\noindent where the ``multiplier"
		\begin{align}
		\mathcal{J}_n:=\frac{1}{2\pi}\int_0^{2\pi}\beta_{n}\left(\gamma_1(s)\right)ds \label{besselop2}
		\end{align}
		is {given explicitly on $F\in \mathcal{H}$ by} \label{multiplier prop}%\mn{tombstone moved}
		\begin{align*}
			\frac{1}{2\pi}\int_0^{2\pi}\beta_{n}\left(\gamma_1(s)\right)F(\zeta)ds
			&=\frac{1}{2\pi}\int_0^{2\pi}\beta_n\left(e^{is}, s/2\right)F\left(\zeta \right)ds\\
			&=\frac{1}{2\pi}\int_0^{2\pi}e^{ins-\sqrt{n}\zeta e^{-is}-\frac{n}{2}}F\left(\zeta +\sqrt{n}e^{is}\right)ds, && z=\sqrt{n}e^{is},\\
			&=\frac{1}{2\pi i }\left(\frac{1}{en}\right)^{n/2}\oint z^{n-1}e^{-n\zeta/z}F\left(\zeta+z\right)dz
		\end{align*}
		which is the same as \eqref{bessel op}.
	\end{proof}

\end{proposition}
\begin{remark}
$\mathcal{J}_n$ is similar to the ``representation integral" considered in \cite{ilmavirta2016radon}, though $s\mapsto \beta_{n}\left(\gamma_1(s)\right)$ is not a homomorphism. Such integration of  representations over geodesics also appear in \cite{guillarmou2017reconstruction}, where the authors used the Principal Series representations of $SL(2, \mathbb{R})$ to show that the normal operator ${I}^*{I}$ associated to the X-ray transform on constant negative curvature surfaces is a nontrivial function of the Laplace-Beltrami operator.
\end{remark}
Together with Proposition \ref{factorication of I}, these imply the Heisenberg Fourier Slice Theorem: 
\begin{proof}[Proof of Theorem \ref{theorem2}]
	Let $f\in L^1(\heis)$, $\lambda>0$ and $n\in\mathbb{Z}^*$. By Proposition \ref{factorication of I} and \ref{multiplier prop} , we have
	\begin{align}
	\mathcal{F}_{\heis/\Gamma}\left({I}_1f\right)(n)
	=\mathcal{F}_{\heis/\Gamma}\left({I}^\text{red}P f\right)(n)
	=(2\pi)\mathcal{J}_n\circ\mathcal{F}_{\heis/\Gamma}\left(P f\right)(n)
	=(2\pi)\mathcal{J}_n\circ\mathcal{F}_\heis f(n).\label{redSlice}
	\end{align}
	Exploiting homogeneity of $I$,
	\begin{align*}
	\mathcal{F}_{\heis/\Gamma_{\lambda}}\left({I}_{\lambda}f\right)(n)
	&=\lambda^{-1}\mathcal{F}_{\heis/\Gamma_{\lambda}}\left(\delta^*_{\lambda}{I}_1\left(\delta^*_{1/\lambda} f\right)\right)(n), & \text{Proposition }\ref{propHomo}\\
	&=\lambda^{-5}\mathcal{F}_{\heis/\Gamma}\left({I}_1\left(\delta^*_{1/\lambda} f\right)\right)(n), &\text{Lemma }\ref{dilation lemma} \\
	&=2\pi \lambda^{-5}\mathcal{J}_n\circ\mathcal{F}_{\heis}\left(\delta^*_{1/\lambda} f\right)(n), & \text{by }\eqref{redSlice},\\
	&=2\pi \lambda^{-1}\mathcal{J}_n\circ\mathcal{F}_{\heis}f(n\lambda^2), & \text{Lemma }\ref{dilation lemma}
	\end{align*}
	as desired.
\end{proof}
\begin{remark}\label{zeroCase}
	In the special case when $n=0$ or $h=0$, the group Fourier transforms are qualitatively different; they are the Euclidean Fourier transform in the $z$ variable (the precise sense in which this limiting behavior occurs is formalized by Geller in \cite{geller1977fourier}). In this case, the Fourier Slice theorem takes the form
	\begin{align*}
	\widetilde{\left(I_\lambda f\right)}(\lambda\zeta, 0)=(2\pi/\lambda)J_0(|\zeta|)\widehat{f}(\lambda\zeta, 0); \quad \forall \lambda>0, \; f\in L^1(\heis),
	\end{align*}
	where $J_0$ is the classical Bessel function of order zero, and
	\begin{align*}
	\widehat{f}(\zeta, 0)=\int_{\mathbb{C}}\int_\mathbb{R}f(z, t)e^{-i\zeta\cdot z}dtdz, \; f\in L^1(\heis), && \widetilde{g}(\zeta, 0)=\int_\mathbb{C}\int_0^{\pi\lambda^{-2}}g(z, t)e^{-i\zeta\cdot z}dtdz; \; g\in L^1(\heis/\Gamma_\lambda).
	\end{align*}
\end{remark}
\subsection{Proof of Theorem \ref{theorem1}}
We now make use of the Heisenberg Fourier Slice theorem to prove injectivity of ${I}$. First, we describe an important class of functions which are the cylindrical harmonics of the Heisenberg group. 

With respect to the standard orthonormal basis {$\{\omega_k(\zeta)=\zeta^k/\sqrt{k!}\in\mathbb{C}: k=0, 1, ...\}$} of $\mathcal{H}$ the matrix coefficients of the Bargmann-Fock representation, \eqref{BFock}, $M^{h}_{jk}(z, t):=\langle \beta_{h}(z, t)\omega_j, \omega_k\rangle_\mathcal{H}$ are given for ${h}>0$ via a brute force computation by
\begin{align}
M^{h}_{jk}(z, t)=
\begin{cases}
\sqrt\frac{k!}{j!} \left(+\sqrt{h}z\right)^{j-k}L_k^{(j-k)}\left(h|z|^2\right)e^{-{h}|z|^2/2}e^{2i{h} t} & j\geq k\\
\sqrt\frac{j!}{k!}\left(-\sqrt{h} \overline{z}\right)^{k-j}L_j^{(k-j)}\left({h}|z|^2\right)e^{-{h}|z|^2/2}e^{2i{h} t} & j \leq k
\end{cases},\label{matrix coefficient}
\end{align}
and $M^{h}_{jk}(z, t)=M^{|{h}|}_{jk}(\overline{z}, -t)$ for ${h}<0$ (see Appendix \ref{alt con} for conversion between Folland's \cite[p. 64]{folland1989harmonic} and our conventions). 

Here $L_j^{(\alpha)}(x)$ is the generalized Laguerre polynomial, defined recursively by
\begin{align}
L_0^{(\alpha)}(x)&=1 \notag\\
L_1^{(\alpha)}(x)&=1+\alpha-x \notag\\
(j+1)L_{j+1}^{(\alpha)}(x)&=\left(2j+1+\alpha-x\right)L_j^{(\alpha)}(x)-\left(j+\alpha\right)L^{(\alpha)}_{j-1}(x).\label{laguerre}
\end{align}
\begin{comment}
\begin{proof}[Proof of Proposition \ref{prop1}.] We compute the SVD of $\mathcal{J}_n$, and find that its singular values do not vanish whenever $n\in2\mathbb{Z}+1$.
	
	$\mathcal{J}_n:\mathcal{H}\to\mathcal{H}$ is bounded in the operator-norm topology for any $n\in\mathbb{Z}^*$ since
	\begin{align*}
	||\mathcal{J}_n||_{\textnormal{op}}\leq\frac{1}{2\pi}\int_0^{2\pi}||\beta_{n}\left(\gamma_1(\theta)\right)||_\textnormal{op}d\theta=1.
	\end{align*}
	For $n>0$,
	\begin{align*}
	\langle \mathcal{J}_n \omega_j, \omega_k\rangle_\mathcal{H}&=\frac{1}{2\pi}\int_0^{2\pi}\langle\beta_{n}\left(e^{i\theta}, \theta/2\right)\omega_j, \omega_k\rangle_\mathcal{H}d\theta\\
	&=\frac{1}{2\pi}\int_0^{2\pi}M^{n}_{jk}\left(e^{i\theta}, \theta/2\right)d\theta\\
	&=\frac{1}{2\pi}\int_0^{2\pi}e^{i(j-k+n)\theta} d\theta M^{n}_{jk}(1, 0) \quad \text{oberving symmetry in }\eqref{matrix coefficient}\\
	&=\delta(j-k+n)M_{jk}^{n}(1, 0)\\
	&=M_{j, j+n}^{n}(1, 0)
	\end{align*}
	in which case, for $n>0$
	\begin{align*}
	\mathcal{J}_n \omega_j %= M^{2n}_{j, j+n}(-i, 0)\omega_{j+n}
	=\sqrt{\frac{j!}{(j+n)!}} \left(\sqrt{n/e}\right)^nL_j^{(n)}(n)\omega_{j+n}.
	\end{align*}
	Furthermore
	\begin{align*}
	\mathcal{J}_{-n}=\frac{1}{2\pi}\int_0^{2\pi}\beta_{n}(e^{-i\theta}, -\theta/2)d\theta=\frac{1}{2\pi}\int_0^{2\pi}\beta_{n}\left(e^{i\theta}, \theta/2\right)d\theta=\mathcal{J}_n.
	\end{align*}
\end{comment}

The following mild generalization of \eqref{besselop2} will be useful for subsequent computations.
	\begin{definition} For $n\in\mathbb{Z}^*$, let
	\begin{align}
	\mathcal{J}_n(r):=\frac{1}{2\pi}\int_0^{2\pi}\beta_{n}\left(re^{i\theta}, \theta/2 \right)d\theta, \quad r>0.\label{besselop3}
	\end{align}
	\end{definition}
In particular, $\mathcal{J}_n(1)=\mathcal{J}_n$, defined in \ref{besselop2}. 
	\begin{proposition}[SVD of $\mathcal{J}_n(r)$]\label{propSVDofJ}
		For every $n\in\mathbb{Z}^*$ and $r>0$, the operator $\mathcal{J}_n(r):\mathcal{H}\to\mathcal{H}$ is bounded in the operator-norm topology. Furthermore, $\mathcal{J}_{-n}(r)=\mathcal{J}_n(r)$, and, with respect to the orthonormal basis $\{\omega_j=\zeta^j/\sqrt{j!}: j=0, 1, 2, ...\}$ of $\mathcal{H}$, we have
		\begin{align}
		\mathcal{J}_n(r)\omega_j=\sqrt{\frac{j!}{(j+n)!}}\left(nr^2\right)^{n/2}e^{-nr^2/2}L_j^{(n)}\left(nr^2\right)\omega_{j+n},\quad \forall j\in\mathbb{N}, \; n>0.\label{J(r)}
		\end{align}
	\end{proposition}
	\begin{proof}
		$\mathcal{J}_n(r):\mathcal{H}\to\mathcal{H}$ is bounded in the operator-norm topology for any $n\in\mathbb{Z}^*$ since
		\begin{align}
		||\mathcal{J}_n(r)||_{\textnormal{op}}\leq\frac{1}{2\pi}\int_0^{2\pi}||\beta_{n}\left(re^{i\theta}, \theta/2\right)||_\textnormal{op}d\theta=1. \label{bounded}
		\end{align}
		Note that, for $n\in\mathbb{Z}^*$,
		\begin{align*}
		\mathcal{J}_{-n}=\frac{1}{2\pi}\int_0^{2\pi}\beta_{n}(e^{-i\theta}, -\theta/2)d\theta=\frac{1}{2\pi}\int_0^{2\pi}\beta_{n}\left(e^{i\theta}, \theta/2\right)d\theta=\mathcal{J}_n.
		\end{align*} 
		For $n>0$,
		\begin{align}
		\langle \mathcal{J}_n(r) \omega_j, \omega_k\rangle_\mathcal{H}&=\frac{1}{2\pi}\int_0^{2\pi}\langle\beta_{n}\left(re^{i\theta}, \theta/2\right)\omega_j, \omega_k\rangle_\mathcal{H}d\theta\nonumber\\
		&=\frac{1}{2\pi}\int_0^{2\pi}M^{n}_{jk}\left(re^{i\theta}, \theta/2\right)d\theta\nonumber\\
		&=\frac{1}{2\pi}\int_0^{2\pi}e^{i(j-k+n)\theta} d\theta M^{n}_{jk}(r, 0) &&\text{oberving symmetry in }\eqref{matrix coefficient}\nonumber\\
		&=\delta(j-k+n)M_{jk}^{n}(r, 0) \label{proofProf1}\\
		&=M_{j, j+n}^{n}(r, 0),\nonumber
		\end{align}
		in which case, 
		\begin{align*}
		\mathcal{J}_n(r) \omega_j %= M^{2n}_{j, j+n}(-i, 0)\omega_{j+n}
		=M_{j, j+n}^n(r, 0)\omega_{j+n},
		\end{align*}
		and, by \eqref{matrix coefficient},%\mn{Parenthesis added}
		\begin{align*}
		M_{j, j+n}^n(r, 0)=\sqrt{\frac{j!}{(j+n)!}}\left(nr^2\right)^{n/2}e^{-nr^2/2}L_j^{(n)}\left(nr^2\right), \quad n>0.&\qedhere
		\end{align*}
	\end{proof}
	\begin{corollary}\label{nonvanishing2}
		Let $r>0$ and $n\in\mathbb{Z}^*$ be fixed. The operator $\mathcal{J}_n(r)$ is injective if and only if $L_j^{(n)}\left(nr^2\right)$ is nonzero for all $j\in\mathbb{N}$.%\mn{Corollary rephrased, and a proof added}
	\end{corollary}
\begin{proof}
	Since $\mathcal{J}_n(r)$ is bounded (by \eqref{bounded}), it is injective if and only if $L_j^{(n)}\left(nr^2\right)$ is nonzero for all $j\in\mathbb{N}$.
\end{proof}
\begin{proposition}
	The {operator} $\mathcal{J}_{n}:\mathcal{H}\to\mathcal{H}$ is injective whenever $n$ is an odd integer. \label{prop1}
\end{proposition}
\begin{proof}%[Proof of Proposition \ref{prop1}]
	Given $n\in2\mathbb{Z}+1$, by Corollary \ref{nonvanishing2}, the operator $\mathcal{J}_n$ is injective if and only if the sequence $\{L_j^{(n)}(n)\}_{j=0}^\infty$ is nonvanishing.
	
	Set $a_j^{(n)}=j!L^{(n)}_j(n)\in\mathbb{Z}$. Then $a^{(n)}_0=a^{(n)}_1=1$, and by \eqref{laguerre},
	\begin{align*}
	a^{(n)}_{j+1}=&(2j+1)a^{(n)}_j-j(j+n)a^{(n)}_{j-1} \\
	=&a^{(n)}_j \quad(\textnormal{mod } 2)
	\end{align*}
	since $n$ is odd. Therefore $a^{(n)}_{j}=a^{(n)}_0=1\; (\text{mod } 2)$ for all $j=0, 1, 2, ...$. In particular, $L_j^{(n)}(n)=a_j^{(n)}/j!\neq 0$ for $j\in\mathbb{N}$. Therefore $\mathcal{J}_n$ is injective whenever $n$ is an odd integer. 
\end{proof}

\begin{remark}\label{laguerreRemark}
	We know {that} $\mathcal{J}_2$ is not injective since $L_2^{(2)}(2)=0$. However, the author is not currently aware of a general statement characterizing all $(j, n)\in\mathbb{N}\times\mathbb{N}^*$ for which $L^{(n)}_j(n)=0$. While knowing this is not essential for proving injectivity of $I$, it would provide more ways to invert $I$. This is because the space of geodesics is four dimensional, and so we only need a subset of the overdetermined data to reconstruct $f$ from $If$.
\end{remark}
The proof of Theorem \ref{theorem1}, {injectivity of the X-ray transform,} is now almost immediate.
\begin{proof}[Proof of Theorem \ref{theorem1}]
	Suppose ${I}_{\lambda}f=0$ for all ${\lambda}\in (0, \eta)$, where $\eta>0$. By the Heisenberg Fourier Slice Theorem (Theorem \ref{theorem2}),
	\begin{align*}
	0=\mathcal{J}_n\circ\mathcal{F}_\heis f(n{\lambda}^2), \quad \forall n\in \mathbb{Z}^*, \; \forall \lambda \in (0, \eta).
	\end{align*}
	By Proposition \ref{prop1}
	\begin{align}
	0=\mathcal{F}_\heis f(n\lambda^2), \quad \forall n\in 2\mathbb{Z}+1, \; \forall\lambda\in(0, \eta).
	\end{align}
	In which case
	\begin{align*}
	0=\mathcal{F}_\heis f({h}), \quad \forall {h}\in \bigcup_{n\in 2\mathbb{Z}+1}n\left(0, \eta^2\right)=\mathbb{R}^*.
	\end{align*}
	Therefore $f=0$ by the Fourier Inversion theorem for $\mathcal{F}_\heis$.
\end{proof}
%\begin{theorem}
%	The Heisenberg X-ray transform ${I}:L^1(\heis)\to L^1(\mathcal{G}, d\nu)$ is injective. 
%\end{theorem}
\section{X-ray transform for the taming metric $g_\epsilon$}\label{taming}
 We use the same machinery to prove injectivity of the X-ray transform associated to the family of left-invariant {\textit{taming metrics}} on $\heis${.} {A taming metric on a sub-Riemannian manifold is a Riemannian metric whose restriction to the horizontal distribution equals the sub-Riemannian metric.} See \cite[Sec. 1.9]{montgomery2002tour}.

Consider the family of left-invariant Riemannian metrics for $\epsilon>0$:
\begin{equation*}
g_{\epsilon} := dx^2+dy^2+\left(1/\epsilon\right)^{2}\Theta^2,
\end{equation*}
where $\Theta :=dt-\frac{1}{2}(xdy-ydx)$ is a contact form for the Heisenberg distribution $\mathcal{D}$, defined in Section \ref{HeisGeo}. {Then $g_\epsilon$ is a taming metric for the sub-Riemannian metric $g=dx^2+dy^2|_{\mathcal{D}}$.  Indeed, since $\mathcal{D}_q=\text{ker }\Theta_q, q\in\heis$, we have $g_\epsilon|_{\mathcal{D}}=g$. }

Geodesics of $(\heis, g_\epsilon)$ converge uniformly to the sub-Riemannian geodesics as $\epsilon\to 0$, \cite[p. 33]{capogna2007introduction}. The explicit expression for $g_\epsilon$ geodesics is derived in \cite[Sec. 2.4.4]{capogna2007introduction}. We record the exponential map for $g_\epsilon$ in \eqref{exp}.

\begin{remark}
	To avoid quantifying $\epsilon$ in every proposition of this section, with the exception of Theorems \ref{slice} {and} \ref{theorem1epsilon}, we will assume that we have chosen a fixed $\epsilon>0$. 
\end{remark}

Let $\mathcal{G}^\epsilon$ be the set of geodesics for $g_\epsilon$ without orientation and $\mathcal{G}^\epsilon_\lambda$ the subset of geodesics having charge $\lambda$ (which is still a constant of motion). Geodesics with $\lambda\neq 0$ still project to circles in the plane, and those with $\lambda=0$ project to lines; {$g_\epsilon$-geodesics} differ from sub-{Riemannian} geodesics {only} by an $\epsilon$-dependent velocity in the $T=\partial_t$ direction. Left translation by any element $(z, t)\in\heis$ is a $g_\epsilon$-isometry, and so $\heis$ acts on $\mathcal{G}^\epsilon$ by pointwise left multiplication. This action does not change the value of $\lambda$ and {is a transitive action on each leaf $\mathcal{G}^\epsilon_\lambda$} when $\lambda\neq 0$. 

We choose a particular geodesic $\gamma_\lambda^\epsilon$ to be the one whose projection to the plane is a unit-speed circular path of radius $R=1/|\lambda|$ {centered at the origin}, and parameterize the set of $g_\epsilon$ geodesics having charge $\lambda$ by
\begin{align}
s\to (z, t)\gamma_\lambda(s), \quad \gamma_\lambda^\epsilon(s)=\left(Re^{is/R}, s\frac{(R^2+2\epsilon^2)}{2R}\right)\in\heis; \quad R=1/\lambda.\label{helix2}
\end{align}
%and parameterize $\mathcal{G}_\lambda^\epsilon$ by $(z, t, \lambda)$
\begin{remark}
	The geodesics described by \eqref{helix2} are not arclength parameterized; indeed,  $g_\epsilon(\dot\gamma_\lambda^\epsilon(s), \dot\gamma_\lambda^\epsilon(s))=1+\epsilon^2\lambda^2$. Instead, we insist that their projections to the plane are unit-speed. 
\end{remark}
We define the X-ray transform associated to the taming metric $g_\epsilon$ by
\begin{align}
I^\epsilon f(z, t, \lambda):=I^\epsilon_\lambda f(z, t):=\int_\mathbb{R} f\left((z, t)\gamma^\epsilon_\lambda(s)\right)ds, \quad f\in C_c(\heis).\label{epsilonXray}
\end{align}

Note that
\begin{align}
\gamma_\lambda^\epsilon(s+2\pi R)
=\gamma_\lambda^\epsilon(s)(0, \pi R^2+2\pi\epsilon^2).\label{stab2}
\end{align}
Therefore the isotropy group of $\gamma_\lambda^\epsilon$ for the action of $\heis$ by left translation on $\mathcal{G}_\lambda^\epsilon$ is
\begin{align*}
\Gamma^\epsilon_\lambda:=\{(0, k\pi(R^2+2\epsilon^2))\in\heis: k\in\mathbb{Z}\}.
\end{align*}
We have the identification
\begin{align*}
\mathcal{G}_{{\lambda}}^\epsilon&\cong\heis/\Gamma_\lambda^\epsilon\\
(z, t)\gamma_{{\lambda}}^\epsilon&\mapsto (z, t)\Gamma_\lambda^\epsilon. 
\end{align*}
\begin{remark}Again, when $\lambda=1$, we omit subscripts and write $\Gamma^\epsilon=\Gamma_1^\epsilon$. We will also write $g(z, t)$, for any function $g:\heis/\Gamma_\lambda^\epsilon\to\mathbb{C}$, in place of $g\left((z, t)\Gamma_\lambda^\epsilon\right)$.
\end{remark}
Let ${d\mu^\epsilon_\lambda(z, t)}\cong dx\wedge dy\wedge dt$ be the Haar measure on $\heis/\Gamma_\lambda^\epsilon$, and let  $\mathcal{G}_{{\lambda}}^\epsilon$ inherit a multiple of the Haar measure %\mn{Placed the following two definitions in align environments}
\begin{align}
d\mathcal{G}_{{\lambda}}^\epsilon:={\lambda} dx\wedge dy\wedge dt. \label{measureEpsilon}
\end{align}
 Furthermore, let 
\begin{align}
d\mathcal{G}^\epsilon:=\lambda e^{-\lambda} dx\wedge dy\wedge dt\wedge d\lambda. \label{measureEpsilon2}
\end{align}
Note the homogeneity of geodesics with respect to dilation:
\begin{align}
\delta_{1/\lambda}\gamma_1^{\epsilon\lambda}(s)=\gamma_\lambda^\epsilon(s/\lambda)
%=&\left(Re^{is/R}, (\tfrac{1}{2}R^2+\epsilon^2)s/R\right)
%=\delta^*_R\left(e^{is/R}, (\tfrac{1}{2}+\epsilon^2/R^2)s/R\right)
; \quad R=1/\lambda. \label{geoHomo2}
\end{align}
\begin{proposition}[Homogeneity of $I^\epsilon$]\label{homoProp2}
For $f\in C_c(\heis)$, we have
\begin{align}
I_\lambda^\epsilon(f)(z, t)=\lambda^{-1}\delta_{\lambda}^*I^{\epsilon\lambda}_1\left(\delta_{1/\lambda}^*f\right)(z, t).\label{homo2}
\end{align}
\end{proposition}
\begin{proof}
	This is essentially the same proof as \eqref{homo}:%\mn{Parenthesis added}
	\begin{align*}
	\delta_{\lambda}^*I^{\epsilon\lambda}_1\left(\delta_{1/\lambda}^*f\right)(z, t)
	&=I^{\epsilon\lambda}_1\left(\delta_{1/\lambda}^*f\right)(\lambda z, \lambda^2 t)\\
	&=\int_\mathbb{R}\delta_{1/\lambda}^*f\left((\lambda z, \lambda^2 t)\gamma_1^{\epsilon\lambda}(s)\right)ds\\
	&=\int_\mathbb{R}f\left((z, t)\delta_{1/\lambda}\gamma_1^{\epsilon\lambda}(s)\right)ds\\
	&=\int_\mathbb{R}f\left((z, t)\gamma_\lambda^{\epsilon}(s/\lambda)\right)ds, & \text{by } \eqref{geoHomo2},\\
	&=\lambda\int_\mathbb{R}\delta_{1/\lambda}^*f\left((\lambda z, \lambda^2
	t)\gamma_1^{\epsilon\lambda}(\lambda s)\right)ds\\
	&=\lambda I_\lambda^\epsilon f(z, t).&&\qedhere
	\end{align*}
\end{proof}
Furthermore, in virtually the same way as Proposition \ref{factorication of I}, we reduce the X-ray transform $I^\epsilon$ to one period:
\begin{proposition}\label{factorication2} For any $\lambda>0$, ${I}_\lambda^\epsilon:L^1(\heis)\to L^1(\mathcal{G}_{{\lambda}}^\epsilon)$ is well-defined, bounded, and factors in the following way:%\mn{Added reminder that $\mathcal{G}_{{\lambda}}^\epsilon\cong \mathbb{H}/\Gamma_\lambda^\epsilon$}
	\[ 
	\begin{tikzcd}
	L^1(\heis) \arrow[d, "P^\epsilon_{\lambda}" left=1]  \arrow[r, "{I}_\lambda^\epsilon"]  &L^1(\mathcal{G}_{{\lambda}}^\epsilon\cong \mathbb{H}/\Gamma_\lambda^\epsilon)\\
	L^1(\heis/\Gamma_\lambda^\epsilon) \arrow[ur, "{I}^{\epsilon, \textnormal{red}}_{\lambda}" below=10, pos=0.75]
	\end{tikzcd}
	\]
	where 
	\begin{align}
	P_\lambda^\epsilon f\left(z, t\right)=\sum_{k\in\mathbb{Z}}f\left(z, t+k\pi(R^2+2\epsilon^2)\right), && I_\lambda^{\epsilon, \textnormal{red}} g(z, t):=\int_0^{2\pi R}g\left((z, t)\gamma_\lambda^\epsilon(s)\right)ds; \, R=1/\lambda.
	\end{align}
	Furthermore, ${I}^\epsilon:L^1(\heis)\to L^1(\mathcal{G}^\epsilon,d\mathcal{G}^\epsilon)$ is well-defined and bounded.
\end{proposition}
\begin{proof}
	By homogeneity \eqref{homo2}, and since pullback by $\delta_\lambda$ is bounded in the above $L^1$ spaces for $\lambda\neq 0$, it suffices to prove the proposition for $\lambda=1$. For this case, we omit subscripts and write $P^\epsilon$ and ${I}^{\epsilon, \textnormal{red}}$.

	For exactly the same reason as \eqref{poisson}, $P^\epsilon$ maps $C_c(\heis)$ to $C_c(\hrede)$, and
	\begin{align}
	\int_{\heis/\Gamma^\epsilon}{P^\epsilon}f\left(z, t\right){d\mu^\epsilon_1(z, t)}=\int_\heis f(z, t)d(z, t). 
	\end{align}
	So in particular, $||{P^\epsilon}f||_{L^1(\heis/\Gamma^\epsilon)}\leq ||f||_{L^1(\heis)}.$
	
	For $g\in C_c(\heis/\Gamma^\epsilon)$, 
	\begin{align*}
	||{I}^{\epsilon, \textnormal{red}}g||_{L^1(\mathcal{G}_1^\epsilon)}
	&=\int_{\mathcal{G}_1^\epsilon}|{I}^{\epsilon, \textnormal{red}}g\left(z, t\right)|d\mathcal{G}_1^\epsilon\\
	&=\int_{\mathbb{H}/\Gamma^\epsilon}\bigg|\int_0^{2\pi}g\left((z, t)\gamma_1^\epsilon(s)\right)ds\bigg|{d\mu^\epsilon_1(z, t)}\\
	&\leq \int_0^{2\pi}\int_{\heis/\Gamma^\epsilon}|g\left((z, t)\gamma_{1}^\epsilon(s)\right)|{d\mu^\epsilon_1(z, t)} ds\\
	&=\int_0^{2\pi}\int_{\heis/\Gamma^\epsilon}|g\left((z, t)\right)|{d\mu_1^\epsilon\left((z, t)\gamma_{1}^\epsilon(s)^{-1}\right)} ds\\
	&=\int_0^{2\pi}\int_{\heis/\Gamma^\epsilon}|g\left(z, t\right)|{d\mu^\epsilon_1(z, t)} ds, && \text{since }\heis/\Gamma^\epsilon\text{ is unimodular,}\\
	&=2\pi||g||_{L^1(\heis/\Gamma^\epsilon)}.
	\end{align*}
	Thus ${P^\epsilon}$ and ${I}^{\epsilon, \textnormal{red}}$ extend to $L^1$ bounded maps. Given $f\in C_c(\heis)$, since $Pf\in C_c(\hrede)$ and
	\begin{align*}
	{I}^{\epsilon, \textnormal{red}}{P^\epsilon}f(z, t)
	&=\int_0^{2\pi}\sum_{k\in\mathbb{Z}}f\left((z, t+k\pi(1+2\epsilon^2))\gamma_1^\epsilon(s)\right) ds 
	=\int_0^{2\pi}\sum_{k\in\mathbb{Z}}f\left((z, t)\gamma_{1}^\epsilon(s+2\pi  k)\right)ds, \quad \text{by }\eqref{stab2}, \\
	&=\sum_{k\in\mathbb{Z}}\int_0^{2\pi}f\left((z, t)\gamma^\epsilon_{1}(s+2\pi  k)\right)ds
	=\sum_{k\in\mathbb{Z}}\int_{2\pi k}^{2\pi (k+1)}f\left((z, t)\gamma^\epsilon_{1}(s)\right)ds={I}_1^\epsilon f(z, t), %\mn{$\gamma_1\to\gamma^\epsilon_{1}$}
	\end{align*}
	we have $||{I}_1^\epsilon f||_{L^1(\mathcal{G}_1^\epsilon)}\leq 2\pi||f||_{L^1(\heis)}$. The third equality follows from uniform convergence of the integrand on the interval $[0, 2\pi]\ni s$. Therefore $I_1^\epsilon$ extends to a bounded map from $L^1(\heis)$ to $L^1(\mathcal{G}_1^\epsilon)$. 
	 In particular{,} one may check, using \eqref{homo2}, that $||I_\lambda^\epsilon f||_{L^1(\mathcal{G}_\lambda^\epsilon)}=||I_1^\epsilon f||_{L^1(\mathcal{G}_1^\epsilon)}\leq 2\pi||f||_{L^1(\heis)}$.
	
	Finally we have, for $f\in L^1(\heis)$,
	\begin{align*}
	||I^\epsilon f||_{L^1(\mathcal{G}^\epsilon)}:=&\int_{\mathcal{G}^\epsilon}|I^\epsilon f(z, t, \lambda)|d{\mathcal{G}^\epsilon}\\
	=&\int_0^\infty\int_{\mathcal{G}_\lambda^\epsilon}|I^\epsilon_\lambda f(z, t)|d\mathcal{G}_\lambda^\epsilon	 e^{-\lambda}d\lambda\\
	\leq& 2\pi||f||_{L^1(\heis)} \int_0^\infty e^{-\lambda}d\lambda=2\pi||f||_{L^1(\heis)}%\mn{$I\to I^\epsilon$}
	\end{align*}
	as desired.
\end{proof}
\begin{remark}
	From these computations, {we} may also deduce a Santal\'{o} formula for $g_\epsilon$:
	\begin{align*}
	\int_{\mathcal{G}_\lambda^\epsilon}I_\lambda^\epsilon f(z, t)d\mathcal{G}_\lambda^\epsilon=2\pi\int_\heis f(z, t)d(z, t), \quad f\in L^1(\heis) %\mn{Period removed}
	\end{align*}
	which refines the usual Santal\'{o} formula. 
\end{remark}
We note a Poisson Summation Formula for $P^\epsilon$:
\begin{lemma} For $f\in L^1(\heis)$, \label{poissonLemma2}
\begin{align}
\mathcal{F}_{\heis/\Gamma^\epsilon}\left(P^\epsilon f\right)(n)
%&=\int_{\hrede}\sum_{k\in\mathbb{Z}}f\left(z, t+k\pi(1+2\epsilon^2)\right)\beta_{\frac{2n}{1+2\epsilon^2}}\left(z, t+k\pi(1+2\epsilon^2)\right)d(z, t)\\
=\fheis f\left(\frac{n}{1+2\epsilon^2}\right), \quad \forall n\in\mathbb{Z}^*.%\mn{Superfluous comma removed}
\end{align}
\end{lemma}
\begin{proof}
	This is just a rescaling of Lemma \ref{summationformula}. Observe that $\Gamma^\epsilon=(1+2\epsilon^2)\Gamma$. Using Lemma \ref{dilation lemma} with $\lambda=1/\sqrt{1+2\epsilon^2}$, and noting that $\delta^*_{\sqrt{1+2\epsilon^2}}P^\epsilon f=P^1\delta^*_{\sqrt{1+2\epsilon^2}}f$, we are done. 
%	 For $F, G \in \mathcal{H}$,
%	\begin{align*}
%	\mathcal\langle\mathcal{F}_{\heis/\Gamma^\epsilon}\left({P^\epsilon}f\right)(n)F, G\rangle_\mathcal{H}:=&\int_{\heis/\Gamma^\epsilon}\sum_{k\in\mathbb{Z}}f\left(z, t+k\pi(1+2\epsilon^2)\right)\langle\beta_{n/(1+2\epsilon^2)}(z, t)^*F, G\rangle_\mathcal{H}{d\mu^\epsilon_1(z, t)} \\
%	=&\int_{\heis/\Gamma^\epsilon}\sum_{k\in\mathbb{Z}}f\left(z, t+k\pi(1+2\epsilon^2) \right)\langle\beta_{n/(1+2\epsilon^2)}(z, t+k\pi(1+2\epsilon^2) )^*F, G\rangle_\mathcal{H}{d\mu_1(z, t)}, \\
%	=&\int_\heis f(z, t)\langle\beta_{n/(1+2\epsilon^2)}(z, t)^*F, G\rangle_\mathcal{H}d(z, t),
%	\end{align*}
%	where the third equality follows from , and the fact that $f(z, t)\langle\beta_{n}(z, t)^*F, G\rangle_\mathcal{H}\in L^1(\heis)$ by the Cauchy-Schwartz inequality. Since $F$ and $G$ were arbitrary, the identity follows from the definition of the Bochner integral.
\end{proof}
Observe how the Fourier transform respects dilations:
\begin{lemma} For $g\in L^1(\heis/\Gamma^\epsilon)$, $\lambda>0$, \label{dilationLemma2}
\begin{align}
\mathcal{F}_{\hredee}\left(\delta_\lambda^* g\right)=\lambda^{-4}\mathcal{F}_\hredeee(g)(n), \quad \forall n\in\mathbb{Z}^*.
\end{align}
\end{lemma}
\begin{proof}
	Observe that $\Gamma^\epsilon_\lambda=\lambda^{-2}(1+2\epsilon^2)\Gamma$, and $\Gamma^{\epsilon\lambda}=(1+2\epsilon^2\lambda^2)\Gamma$. Then apply Lemma \ref{dilation lemma}. 
\end{proof}
%so that
%\begin{align}
%\beta_1(re^{i\theta}, 0)=\sum_{n\in\mathbb{Z}}\mathcal{J}_n(r)e^{-in\theta}
%\end{align}

As before, $I^{\epsilon, \textnormal{red}}$ is a convolution operator by a compactly supported distribution. We compute its {generalized} Fourier multiplier:
\begin{proposition}\label{multProp}
	For $g\in L^1(\heis/\Gamma^\epsilon)$, 
	\begin{align*}
	\mathcal{F}_{\hrede}\left(I^{\epsilon, \textnormal{red}} g\right)(n)=2\pi\mathcal{J}_n\left(\frac{1}{\sqrt{1+2\epsilon^2}}\right)\circ\mathcal{F}_\hrede(g)(n), \quad \forall n\in\mathbb{Z}^*.
	\end{align*}
\end{proposition}
\begin{proof}
\begin{align*}
\mathcal{F}_{\hrede}\left(I^{\epsilon, \textnormal{red}} g\right)(n)
&=\int_\hrede \int_0^{2\pi}g\left((z, t)\gamma_1^\epsilon (s)\right)\beta_{n/(1+2\epsilon^2)}(z, t)^*ds{d\mu^\epsilon_1(z, t)} \\
=&\int_0^{2\pi}\int_{\heis/\Gamma^\epsilon}g\left(z, t\right)\beta_{n/(1+2\epsilon^2)}\left((z, t)\gamma_1^\epsilon(s)^{-1}\right)^*{d\mu^\epsilon_1(z, t)} ds\\%, & \text{since }\heis/\Gamma^\epsilon\text{ is unimodular,}\\
=&\int_0^{2\pi }\int_{\heis/\Gamma^\epsilon}g\left(z, t\right)\beta_{n/(1+2\epsilon^2)}\left(\gamma_1(s)\right)\circ\beta_{n/(1+2\epsilon^2)}(z, t)^*{d\mu^\epsilon_1(z, t)} ds\\%, &\text{since }\beta_{n/(1+2\epsilon^2)}(z, t)\text{ is unitary},\\
&=\int_0^{2\pi}\beta_{n/(1+2\epsilon^2)}\left(\gamma_1^\epsilon (s)\right)ds \circ \mathcal{F}_\hrede(g)(n).\\
&=:2\pi\mathcal{J}_n\left(\frac{1}{\sqrt{1+2\epsilon^2}}\right)\circ\mathcal{F}_\hrede(g)(n).
\end{align*}
\end{proof}
We may now prove the Heisenberg Fourier Slice Theorem for $g_\epsilon$:
%\begin{theorem}[$g_\epsilon$ Heisenberg Fourier Slice Theorem]\label{slice} If $f\in L^1(\heis)$, and $\epsilon>0$ then
%	\begin{align*}
%	\mathcal{F}_\hredee\left(I_\lambda^\epsilon f\right)(n)=(2\pi/\lambda)\mathcal{J}_n\left(\frac{1}{\sqrt{1+2\epsilon^2\lambda^2}}\right)\circ\fheis(f)\left(\frac{n\lambda^2}{1+2\epsilon^2\lambda^2}\right), \quad \forall n\in\mathbb{Z}^*, \;\forall \lambda>0.
%	\end{align*}
%\end{theorem}
\begin{proof}[Proof of Theorem \ref{slice}]

Combining Proposition \ref{factorication2} and \ref{multProp},
\begin{align}
\mathcal{F}_\hrede\left(I^\epsilon f\right)%=2\pi\mathcal{J}_n\left(\frac{1}{\sqrt{1+2\epsilon^2}}\right)\circ\mathcal{F}_\hrede\left(P^\epsilon f\right)(n)
=2\pi\mathcal{J}_n\left(\frac{1}{\sqrt{1+2\epsilon^2}}\right)\circ\fheis f\left(\frac{n}{1+2\epsilon^2}\right).\label{redSlice2}
\end{align}
Now, exploiting homogeneity of $I^\epsilon$,
\begin{align*}
\mathcal{F}_\hredee\left(I_\lambda^\epsilon f\right)(n)
&=\lambda^{-1}\mathcal{F}_\hredee\left(\delta_{\lambda}^*I^{\epsilon\lambda}_1\left(\delta_{1/\lambda}^*f\right)\right)(n), & \text{by Propositoin } \ref{homoProp2},\\
&=\lambda^{-5}\mathcal{F}_\hredeee\left(I_1^{\epsilon\lambda}\left(\delta_{1/\lambda}^*f\right)\right)(n), & \text{by Lemma } \ref{dilationLemma2},\\
&=2\pi \lambda^{-5}\mathcal{J}_n\left(\frac{1}{\sqrt{1+2\epsilon^2\lambda^2}}\right)\circ\fheis(\delta_{1/\lambda}^*f)\left(\frac{n}{1+2\epsilon^2\lambda^2}\right), & \text{by } \eqref{redSlice2},\\
&=(2\pi/\lambda)\mathcal{J}_n\left(\frac{1}{\sqrt{1+2\epsilon^2\lambda^2}}\right)\circ\fheis f\left(\frac{n\lambda^2}{1+2\epsilon^2\lambda^2}\right), &\text{by Lemma } \ref{dilationLemma2}.&\qedhere
\end{align*}
\end{proof}
\begin{proposition}\label{propInjective}
	Let $\epsilon>0$ and $n\in\mathbb{Z}^*$ be fixed. Then  $\mathcal{J}_n\left(\frac{1}{\sqrt{1+2\epsilon^2\lambda^2}}\right): \mathcal{H}\to\mathcal{H}$ is injective for almost all $\lambda>0$.
\end{proposition}
\begin{proof}
	Set $r=\frac{1}{\sqrt{1+2\epsilon^2\lambda^2}}$. By Corollary \ref{nonvanishing2}, the operator $\mathcal{J}_n(r)$ is injective if and only if $nr^2$ is not a zero of $L_j^{(n)}$ for any $j\in\mathbb{N}$. Since there are only countably many such zeros, the proposition follows. 
\end{proof}
%From \eqref{J(r)},
%\begin{align}
%\mathcal{J}_n\left(\frac{1}{\sqrt{1+2\epsilon^2\lambda^2}}\right)\omega_j=\sqrt{\frac{j!}{(j+n)!}}\left(\frac{n}{1+2\epsilon^2\lambda^2}\right)^{n/2}e^{-\frac{n/2}{1+2\epsilon^2\lambda^2}}L_j^{(n)}\left(\frac{n}{1+2\epsilon^2\lambda^2}\right)\omega_{j+n}.\label{J(r)}
%\end{align}
%\begin{theorem}\label{theorem1epsilon}
%	For all $\epsilon>0$, the Heisenberg taming X-ray transform $I^\epsilon: L^1(\heis)\to L^1(\mathcal{G}^\epsilon, d\mathcal{G}^\epsilon)$ is injective. In particular, if $f\in L^1(\heis)$ and $I_\lambda^\epsilon f=0$ for all $\lambda$ in a neighborhood of zero, then $f=0$. 
%\end{theorem}
%By Proposition \ref{propInjective}, for every $n\in\mathbb{Z}^*$, there exists a full-measure set $S_n\subset(0, \infty)$ for which the operator
%\begin{align*}
%\mathcal{J}_n\left(\frac{1}{\sqrt{1+2\epsilon^2\lambda^2}}\right)
%\end{align*}
%is injective for all $\lambda\in S_n$. 

We now have the tools to prove injectivity of the taming X-ray transform $I^\epsilon$:
\begin{proof}[Proof of Theorem \ref{theorem1epsilon}]
	Suppose, $I_{\lambda}^\epsilon f=0$ for all $\lambda\in (0, \eta)$, where $\eta>0$. Then by Theorem \ref{slice} and Proposition \ref{propInjective}, 
	\begin{align*}
	0=\fheis f\left(\frac{n\lambda^2}{1+2\epsilon^2\lambda^2}\right)%\label{theorem4proof}
	\end{align*}
	for almost all $\lambda\in(0, \eta)$, and all $n\in\mathbb{Z}^*$. Let $A$ be the set of all such $\lambda\in(0, \eta)$, and $B=\{\lambda^2/(1+2\epsilon^2\lambda^2): \lambda\in A\}$. Then in other words
	\begin{align*}
	0=\fheis f(h) && \forall h\in \bigcup_{n\in\mathbb{Z}} nB.%\mn{$\infty$ removed}
	\end{align*}
	Since $B$ has full measure on the interval $\left(0, \frac{\eta^2}{1+2\epsilon^2\eta^2}\right)$, we know $\fheis{f}=0$ almost everywhere. Therefore $f=0$ by the Fourier Inversion Theorem. %\mn{Missing proof environment added. Superfluous dash and parenthesis removed.}
\end{proof}

\begin{comment}
By \eqref{J(r)}, we know that $\mathcal{J}_n(r)$ is injective if and only if $nr^2$ is not a root of $L_j^{(n)}$ for all $j=0, 1, 2, ..$. If, for example, $\lambda_k=\sqrt{\pi/k\epsilon}$, then $nr^2=\frac{n}{1+2\pi /k}$ is transcendental for all $k=1, 2, 3, ...$. Thus 
\begin{align*}
\mathcal{J}_n\left(\frac{1}{\sqrt{1+2\epsilon^2\lambda_k^2}}\right)
\end{align*}
is injective for all $n\in\mathbb{Z}^*$ and all $k=1, 2, ...$. Therefore, by \eqref{slice}, if $I_{\lambda}^\epsilon(f)=0$ for all $\lambda$ in a neighborhood of zero, then
\begin{align}
0=\fheis(f)\left(\frac{n\lambda_k^2}{1+2\epsilon^2\lambda_k^2}\right)
\end{align}
for all $n\in\mathbb{Z}^*$ and $k\geq K$, with $K\in\mathbb{N}$ sufficiently large. In other words
\begin{align}
0=\fheis(f)(h) && \forall h\in S:=\bigcup_{k\geq K}^\infty \frac{\lambda_k^2}{1+2\epsilon^2\lambda_k^2}\mathbb{Z}^*.
\end{align}
Since $f\in L^1(\heis)$, $\fheis{f}$ is continuous as a function of $h$. Since $S$ is dense in $\mathbb{R}^*$, $f=0$ be the Fourier Inversion Theorem. 
\end{comment}
\section{Appendix}
\subsection{SVD of $I^\textnormal{red}|_{^0L^2(\hred)}$}\label{redSVD}
While not strictly necessary for our main result, the computation in Proposition \ref{propSVDofJ} also gives us the SVD {of} $\xred$ when restricted to a specific subspace. Here, similarly with \cite{monard2019functional}, we implicitly exploit the fact that $\xred$ intertwines the Heisenberg Laplacian  on $\heis$ with another differential operator on $\heis/\Gamma$ for which the functions $M_{jk}^h$, $h\in\mathbb{R}^*$, and $M_{jk}^{n}$, $n\in\mathbb{Z}^*$, are eigenfunctions, respectively. 

Consider the subspaces of $L^2(\hred)$
\begin{align*}
L^2(\mathbb{C})\cong&\{f\in L^2(\hred): f(z, t)=f(z, 0),\; \forall(z, t)\in\hred \}\\
 ^0L^2(\hred):=&\{f\in L^2(\hred): \int_0^\pi f(z, t)dt=0, \; \forall z\in\mathbb{C} \}.
\end{align*}

\begin{lemma} We have the orthogonal decomposition\label{orthonormal}
	\begin{align}
	L^2(\heis/\Gamma)\cong L^2(\mathbb{C})\oplus {^0L^2(\hred)}. \label{orthoDecomp}
	\end{align}
\end{lemma}
\begin{proof}
	{Given} $f\in L^2(\hred)$, let 
	\begin{align*}
	f_0(z, t):=\frac{1}{\pi}\int_0^{\pi}f(z, t)dt && \text{and} && g=f-f_0.
	\end{align*}
	 {Then} $f_0\in L^2(\mathbb{C})$ and $g\in {^0}L^2(\hred)$. 
	
	Furthermore, for arbitrary $f_0\in L^2(\mathbb{C})$, and $g\in {^0}L^2(\hred)$, 
	\begin{align*}
	\int_\hred f_0(z, t)\overline{g(z, t)}{d\mu_1(z, t)}=\int_\mathbb{C}f_0(z)\int_0^\pi \overline{g(z, t)}dtdz=0.
	\end{align*}
	The orthogonal decomposition \eqref{orthoDecomp} follows.
\end{proof}
In what follows, set 
\begin{align}
\psi^n_{jk}:=\frac{\sqrt{2|n|}}{2\pi}M^{n}_{jk}; \quad j, k\in\mathbb{N},\, n\in\mathbb{Z}^*
\end{align} %\mn{The theorem environment restating this line has been removed. Remark about convolutions also removed.} 
for $M_{jk}^n$ defined in \eqref{matrix coefficient}. The functions $\psi^n_{jk}$ {for, $n\in\mathbb{Z}^*$ and $j, k\in\mathbb{N}$,} form an orthonormal basis for ${^0L^2(\hred)}$. (See \cite[Ch. 4]{thangavelu2012harmonic}, where the author uses slightly different notation.) 
%\begin{theorem*} 
%	Let $f\in L^2(\hred)$ such that $\int_0^{\pi} f(z, t)dt=0$, then we have
%	\begin{align}
%	f=\sum_{n\in\mathbb{Z}^*}\sum_{j, k=0}^\infty \langle f, \psi_{jk}^n\rangle_{L^2(\hred)}\psi^n_{jk}
%	\end{align}
%	where the series converges in $L^2$. See \cite[Ch. 4]{thangavelu2012harmonic}
%\end{theorem*}
%The next result is a consequence of the fact that $\xred$ is convolution by a compactly supported 
% distribution, but we proceed explicitly. 
\begin{proposition}
	\begin{align*}
	I^\textnormal{red}: L^2(\heis/\Gamma)\to L^2(\heis/\Gamma)
	\end{align*}
	is well-defined and bounded. 
\end{proposition}
\begin{proof} For $g\in C_c(\hred)$, the Cauchy-Schwartz Inequality yields
	\begin{align}
	|\xred g(z, t)|^2=\left(\int_0^{2\pi}|g\left((z, t)(e^{i\theta}, \theta/2)\right)|d\theta\right)^2\leq 2\pi%\mn{Corrrected a mistake in the proof. Cauchy-Schwarts needed, not the MVT}
	 \int_0^{2\pi}|g\left((z, t)(e^{i\theta}, \theta/2)\right)|^2d\theta. \label{reducedBounded}
	\end{align}
	Then
	\begin{align*}
	||\xred g ||^2_{L^2(\hred)}&=\int_\hred |\xred g(z, t)|^2{d\mu_1(z, t)}\\
	&\leq(2\pi)\int_0^{2\pi}\int_\hred |g\left((z, t)(e^{i\theta}, \theta/2)\right)|^2{d\mu_1(z, t)}d\theta, && \text{by }\eqref{reducedBounded}, \\
	&=(2\pi)^2\int_\hred |g\left((z, t)\right)|^2{d\mu_1(z, t)}, && {\text{by left-invariance of }\mu_1},\\
	&=(2\pi)^2||g||^2_{L^2(\hred)},
	\end{align*}
	so $\xred$ extends to a bounded function from $L^2(\hred)$ to itself. %\mn{``on" removed}
\end{proof}
\begin{proposition}\label{decomp}
	$\xred$ preserves the orthogonal decomposition in Lemma \ref{orthonormal}. i.e{,} %\mn{``." replaced by ``,"}
	\begin{align*}
	&\xred|_{L^2(\mathbb{C})}: L^2(\mathbb{C})\to L^2(\mathbb{C})\\
	&\xred|_{{^0}L^2(\hred)}: {^0}L^2(\hred)\to {^0}L^2(\hred).
	\end{align*}
	Furthermore, the restriction $I^\textnormal{red}|_{L^2(\mathbb{C})}$ is essentially $2\pi$ times the Mean Value Transform $M^1$. 
\end{proposition}
\begin{proof}
	For $f\in L^2(\mathbb{C})$, 
	\begin{align*}
	I^\textnormal{red}|_{L^2(\mathbb{C})}f(z, t)
	&=\int_0^{2\pi}f\left((z, t)(e^{i\theta}, \theta/2)\right)d\theta
	=\int_0^{2\pi}f\left(z+e^{i\theta}, 	t+\theta/2+\tfrac{1}{2}\textnormal{Im}\left(\overline{z}e^{i\theta}\right)\right)d\theta\\
	&=\int_0^{2\pi}f(z+e^{i\theta})d\theta=2\pi M^1f(z), 
	\end{align*}
	and so $\xred f\in L^2(\mathbb{C})$.
	
	For $g\in {^0}L^2(\hred)$, 
	\begin{align*}
	\int_0^\pi \xred g(z, t)dt=\int_0^\pi \int_0^{2\pi}g\left(z+e^{i\theta}, 	t+\theta/2+\tfrac{1}{2}\textnormal{Im}\left(\overline{z}e^{i\theta}\right)\right)d\theta dt=\int_0^{2\pi}\int_0^\pi g(z+e^{i\theta}, t)dtd\theta=0,
	\end{align*}
	so that $\xred g\in {^0}L^2(\hred)$. 
\end{proof}
We know that $\xred|_{L^2(\mathbb{C})}=2\pi M^1$ has a continuous spectrum (see \eqref{radon}, or Remark \ref{zeroCase}), so we restrict the reduced X-ray transform to ${^0L^2(\hred)}$, where {it} has a discrete spectrum, and compute the Singular Value Decomposition there. 

\begin{theorem}[SVD of $I^\text{red}|_{^0L^2(\hred)}$]\label{SVD}
	For all $n\in\mathbb{Z}^*$ and $j, k\in \mathbb{N}$, 
	\begin{align*}
	I^\textnormal{red}|_{^0L^2(\hred)}\psi^n_{jk}=2\pi \sqrt{\frac{j!}{(j+|n|)!}}\left(|n|/e\right)^{|n|/2}L_j^{(|n|)}(|n|)\psi^n_{j+|n|, k}. %\mn{Scentence restructured}
	\end{align*}
\end{theorem}
\begin{proof}
	Note that, for $(w, s), (z, t)\in\mathbb{H}$
	\begin{align*}
	M_{jk}^{n}\left((w, s)(z, t)\right)%\mn{$g_1, g_2$ replaced by $(w, s)$ and $(z, t)$ respectively}
	&=\langle\beta_{n}\left((w, s)(z, t)\right)\omega_j, \omega_k\rangle_\mathcal{H}
	=\langle\beta_{n}(w, s)\circ\beta_{n}(z, t)\omega_j, \omega_k\rangle_\mathcal{H}\\
	&=\sum_{l=0}^\infty\langle\beta_{n}(w, s)\omega_l, \omega_k\rangle_\mathcal{H}\langle\beta_{n}(z, t)\omega_j, \omega_l\rangle_\mathcal{H}=\sum_{l=0}^\infty M_{jl}^{n}(z, t)M_{lk}^{n}(w, s).
	\end{align*}
	{Then}
	\begin{align*}
	\xred|_{^0L^2(\hred)}\psi_{jk}^n(z, t)%\mn{$M_{jk}^{2n}$ corrected to $M_{jk}^n$}
	&=\frac{\sqrt{2|n|}}{2\pi}\int_0^{2\pi}M^{n}_{jk}\left((z, t)(e^{i\theta}, \theta/2)\right)d\theta\\
	&=\frac{\sqrt{2|n|}}{2\pi}\sum_{l=0}^\infty\int_0^{2\pi}M^{n}_{jl}(e^{i\theta}, \theta/2)M_{lk}^{n}(z, t) d\theta\\
	&=\frac{\sqrt{2|n|}}{2\pi}\sum_{l=0}^\infty \delta(j-l+|n|)M^{n}_{jl}(1, 0)M^{n}_{lk}(z, t), &&\text{by } \eqref{proofProf1}\text{ in Proposition }\ref{prop1},\\
	&=M^{n}_{j, j+n}(1, 0)\psi^n_{j+|n|, k}(z, t)\\
	&=2\pi \sqrt{\frac{j!}{(j+|n|)!}}\left(|n|/e\right)^{|n|/2}L_j^{(|n|)}(|n|)\psi^n_{j+|n|, k}(z, t).&&\qedhere
	\end{align*}
\end{proof} 
In view of Proposition \ref{decomp} and Theorem \ref{SVD}, the kernel of $I^\text{red}$ on %\mn{Added the following remark about the kernel of $I^\text{red}$}
 $L^2(\heis/\Gamma)$ is given by the $L^2$-closure of
\begin{align*}
	\text{Span}\{\psi_{jk}^n: j, k\in\mathbb{N}, n\in\mathbb{Z}^*, L^{(|n|)}_j(|n|)=0\}
\end{align*}
We know this kernel contains at least the closure of $\{\psi_{2, k}^2: k=0, 1, 2...\}$ since $L_2^{(2)}(2)=0$. Determining the entire kernel will require a number-theoretic argument (see Remark \ref{laguerreRemark}).
\subsection{Exponential Map for Heisenberg Geodesics}
The sub-Riemannian flow maps from the unit cotangent bundle $U^*\heis:=H^{-1}(\frac{1}{2})$ to itself. We work in the left-trivialization of the unit cotangent bundle: $U^*\heis\cong \heis \times U(1)\times\mathbb{R}\ni (z, t, e^{i\phi}, \lambda)$. The exponential map $\exp: \mathbb{R}\times U^*\heis\to \heis$ is given in these coordinates by
\begin{equation}
\exp_{(z, t)}\left(s(e^{i\phi}, \lambda)\right) =(z, t)
\begin{cases}
 \left(e^{i\phi}\frac{(e^{i\lambda s}-1)}{i\lambda}, \, \frac{\lambda s-\sin(\lambda s)}{2\lambda^2}\right) & \lambda\neq 0 \\
\left(se^{i\phi}, 0\right) & \lambda=0
\end{cases}\label{sRexp}%\mn{Period removed}
\end{equation}
(see \cite[Ch. 1]{montgomery2002tour}){.} {As a function of $s$, this} describes the unit-speed geodesic with initial point $(z, t)$ whose projection to the plane is a counterclockwise{-parameterized} circle of radius $R=1/|\lambda|$ {with} initial velocity in the direction of $\phi$ if $\lambda>0$, and $\phi+\pi$ if $\lambda<0$. If $\lambda=0$ {the} projection is a {straight} line in the direction $\phi$. The geodesics in \eqref{geod} are obtained by rotations and left translation of \eqref{sRexp}.

The Riemannian exponential map $\exp^\epsilon$ for $g_\epsilon$ is given in the same coordinates by
\begin{equation}
\exp^\epsilon_{(z, t)}\left(s(e^{i\phi}, \lambda)\right) =\exp_{(z, t)}\left(s(e^{i\phi}, \lambda)\right)(0, \epsilon^2\lambda s)\label{exp}
\end{equation}
(see \cite[Thm. 11.8]{montgomery2002tour} for an explanation).
Because we are using cylindrical coordinates in the fibers, neither of these exponential maps describe geodesics with initial condition strictly in the $\lambda$ direction. In the case of $g$, these geodesics are fixed points in $\heis$, and in the case of $g_\epsilon$ these geodesics are integral curves of the Reeb vector field $\epsilon^2\lambda T$. In both cases, the X-ray transforms are inverted without considering these geodesics. 
\subsection{Bessel Functions}
The classical Bessel function of order $n$ is defined by
\begin{align}
J_n(r):=\frac{1}{2\pi i^n}\int_0^{2\pi}e^{ir\cos{\theta}}e^{-in\theta}d\theta.\label{bessel}
\end{align}
\subsection{Infinitesimal Representation}\label{infinitesimal}
Define the complex vector fields on $\heis$:
\begin{align*}
Z:=\frac{1}{2}\left(X-iY\right), && \overline{Z}:=\frac{1}{2}\left(X+iY\right)
\end{align*}
where $X$ and $Y$ are given in \eqref{XY}. Then $\beta_h: \heis\to\mathcal{U}(\mathcal{H})$ {as defined in} \eqref{BFock} is the unique strongly continuous unitary {group homomorphism}  
%of $\heis$ on $\mathcal{H}$ 
for which, on the level of Lie algebras,
\begin{align*}
\left(\beta_h\right)_*Z=\sqrt{h}\partial_\zeta, && \left(\beta_h\right)_*\overline{Z}=-\sqrt{h}\zeta, && (\beta_h)_*T=2h. 
\end{align*}
Fix $F\in\mathcal{H}$ and $(z, t)\in\heis$. To obtain \eqref{BFock}, let %\mn{Rewrote the second half of this section and added a reference}
%$G_h(\tau, \zeta):=\beta_h(\tau(z, t))F(\zeta)$ for $F\in\mathcal{H}$ and $(z, t)\in\heis$ then
 $G_h(\tau, \zeta)$ be unique solution to the differential equation
\begin{align*}
%ZG=\sqrt{h}\partial_\zeta G && \overline{Z}G=-\sqrt{h}\zeta G && TG=2hG
\frac{d}{d\tau}G_h(\tau, \zeta)=(\beta_h)_*\left(tT+zZ+\overline{z}\overline{Z}\right)G_h(\tau, \zeta)=\left(2iht+\sqrt{h}(z\partial_\zeta-\overline{z}\zeta)\right)G_h(\tau, \zeta)
\end{align*}
subject to the condition $G_h(0, \zeta)=F(\zeta)$. Then $\beta_h(z, t)F(\zeta):=G_h(1, \zeta)$. See \cite[Ch. 1 Sec 3]{folland1989harmonic} to see this worked out for the Schr\"{o}dinger representation.

%Note that $\zeta$ and $\partial_\zeta$ are the creation and annihilation operators on $\mathcal{H}$. One may check that
%\begin{align*}
%\beta_h(z, t)=e^{2iht+\sqrt{h}\left(z\partial_\zeta-\overline{z}\zeta\right)}
%\end{align*}
%is the same as \eqref{BFock}.
\subsection{Alternate Conventions}\label{alt con}
Folland \cite{folland1989harmonic} defines the Bargmann-Fock representation on the 1-parameter family of Hilbert spaces 
\begin{align*}
\mathcal{H}^{h}:=\bigg\{F:\mathbb{C}\to\mathbb{C}, \text{ holomorphic }:%\mn{Erroneous ``u" changed to ``F."}
 h\int_\mathbb{C}|F(\zeta)|^2e^{-\pi h |\zeta|^2}d\zeta<\infty \bigg\}, \quad h>0,
\end{align*} 
and $\mathcal{H}^h:=\big\{F: \overline{F}\in\mathcal{H}^{|h|}\big\} $ for $h<0$. 

For $h\in\mathbb{R}^*$ and $\lambda>0$, the maps
\begin{align*}
S_\lambda:\mathcal{H}^h&\to\mathcal{H}^{\lambda h}; & S(F)(\zeta):=&F(\sqrt{\lambda}\zeta)\\
c:\mathcal{H}^h&\to\mathcal{H}^{-h};& c(F):=&\overline{F}
\end{align*}
are all isometries.

Folland defines the Fock representation, for $h>0$,  as
\begin{align*}
\beta^\text{Fol}_h(z, t)F(\zeta):=e^{2\pi hit-\pi h\zeta\overline{z}-\pi h|z|^2/2}
F(\zeta+z), \quad F\in\mathcal{H}^h
\end{align*}
and $\beta^\text{Fol}_h(z, t)=c\circ\beta^\text{Fol}_{|h|}\left(\overline{z}, -t\right)\circ c$ for $h<0$.

Our definition is rescaled so that every $\beta_{h}$ acts on the same space $\mathcal{H}=\mathcal{H}^{1/\pi}$. {Folland's} definition, $\beta^\text{Fol}_h$, is related to ours via
\begin{align*}
\beta^\text{Fol}_h(z, t)=
%\begin{cases}
S_{\pi h}\circ\beta_{\pi h}\left(z, t\right)\circ S_{\pi h}^{-1}, \quad h>0. %& h>0 \\
%c^*\circ S^*_{\sqrt{\pi |h|}}\circ\beta_{2\pi h}\left(w, s\right)\circ S^*_{\sqrt{\pi |h|}^{-1}}\circ c* & h<0
%\end{cases}
\end{align*}
An advantage of this convention is that as $h$ varies, $\beta_h$ varies by precomposition with automorphisms of $\heis$:
\begin{equation*}
\begin{array}{l}
\beta_h(z, t)=\beta_1(\sqrt{h}z, ht), \quad \text{for }h>0\\
\beta_h(z, t)=\beta_{|h|}(\overline{z}, -t), \quad \text{for }h<0.
\end{array}
\end{equation*}
Granted, an advantage of Folland's definition is that the Fourier transform defined with $\beta_h^\text{Fol}$ does ``converge" to the Euclidean Fourier transform as $h\to 0$. 
\section*{Acknowledgements}
%The author acknowledges partial support from NSF grant DMS-1814104.\\\\
The author gratefully acknowledges the support of advisors Richard Montgomery and Fran\c{c}ois Monard, who provided significant guidance and insight for the duration of this project.\\\\
This material is based upon work supported by the National Science Foundation under grants DMS-1814104, and DMS-1440140 while the author was in residence at the Mathematical Sciences Research Institute in Berkeley, California, during the Fall 2019 semester.\\\\
The author thanks the anonymous reviewer who provided many helpful comments.
\bibliographystyle{amsplain}
\bibliography{ReferencesV10}
\end{document}